\documentclass[english,12pt,oneside]{amsproc}
\usepackage[T2A]{fontenc}
\usepackage[english]{babel}
\usepackage{sseq}
\usepackage{graphics}
\usepackage{amsfonts, amssymb, amscd, amsmath}
\usepackage{latexsym}
\usepackage[matrix,arrow,curve]{xy}
\usepackage{mathabx,mathtools}
\usepackage{color}
\usepackage{mathrsfs}
\usepackage{mathdots}
\usepackage{pigpen}
\usepackage{tikz}
\usetikzlibrary{matrix}

\DeclareMathOperator{\lk}{lk} \DeclareMathOperator{\cone}{Cone}
 \DeclareMathOperator{\Ker}{Ker}
\DeclareMathOperator{\coker}{Coker} \DeclareMathOperator{\id}{id}
\DeclareMathOperator{\codim}{codim}

\DeclareMathOperator{\sgn}{sgn}
 \DeclareMathOperator{\im}{Im}
 
\DeclareMathOperator{\ver}{Vert} \DeclareMathOperator{\Tot}{Tot}
\DeclareMathOperator{\tr}{tr}

\newcommand{\Zo}{\mathbb{Z}}
\newcommand{\Ro}{\mathbb{R}}
\newcommand{\Co}{\mathbb{C}}
\newcommand{\Qo}{\mathbb{Q}}

\newcommand{\ko}{\Bbbk}
\newcommand{\Zt}{\Zo_2}

\newcommand{\eqd}{\stackrel{\text{\tiny def}}{=}}

\newcommand{\lco}{^\lfloor}
\newcommand{\rco}{^\rfloor}

\newcommand{\ca}[1]{\mathcal{#1}}
\newcommand{\cah}[1]{\widehat{\mathcal{#1}}}
\newcommand{\ups}[1]{{\lco #1\rco}}
\newcommand{\sta}[1]{(\ast_{#1})}
\newcommand{\inc}[2]{[#1:#2]}

\newcommand{\ang}{\raisebox{-2pt}{\pigpenfont G}}

\newcommand{\ld}{\texttt{N}}

\newcommand{\A}{\ca{A}}
\newcommand{\R}{\ca{R}}
\newcommand{\Cc}{C} 
\newcommand{\Tt}{\ca{T}}
\newcommand{\I}{\ca{I}}

\newcommand{\La}{\ca{L}}
\newcommand{\loc}{\ca{U}}
\newcommand{\Lah}{\widehat{\ca{L}}}
\newcommand{\Pih}{\widehat{\Pi}}

\newcommand{\Iseq}{{_{seq}\!\I}}
\newcommand{\Rseq}{{_{seq}\!R}}
\newcommand{\Cseq}{{_{seq}\!C}}

\newcommand{\Hr}{\widetilde{H}}
\newcommand{\dd}{\partial}
\newcommand{\hh}{\mathcal{H}}
\newcommand{\hht}{\underline{\mathcal{H}}}

\newcommand{\oo}{^{\circ}}

\newcommand{\Ey}{{{^Y}\!E}}
\newcommand{\Eyt}{{{^{\dd Y}}\!E}}
\newcommand{\Ex}{{{^X}\!E}}
\newcommand{\Ext}{{{^{\dd X}}\!E}}
\newcommand{\Ept}{{{^{\dd P}}\!E}}
\newcommand{\Ep}{{{^{P}}\!E}}
\newcommand{\Eh}{{{^H}\!E}}
\newcommand{\Ev}{{{^V}\!E}}
\newcommand{\Eq}{{{^Q}\!E}}
\newcommand{\Eqt}{{{^{\dd Q}}\!E}}

\newcommand{\Eyj}{{{^Y_j}\!E}}
\newcommand{\Exj}{{{^X_j}\!E}}
\newcommand{\Ccq}{{{_q}\Cc}}

\newcommand{\dy}{d_Y}
\newcommand{\dx}{d_X}
\newcommand{\dho}{d^H}
\newcommand{\dv}{d^V}
\newcommand{\dqq}{d_Q}

\newcommand{\simc}{\!\!\sim}

\newcommand{\ft}{\widehat{f}}
\newcommand{\chir}{\widetilde{\chi}}
\newcommand{\br}{\widetilde{b}}

\newcommand{\ST}[1]{\mbox{\upshape\small #1}}
\newcommand{\cat}{\ST{CAT}}
\newcommand{\catop}{\ST{CAT}^{op}}

\newcommand{\MOD}{\ST{MOD}}

\newcommand{\ddb}{b_1}

\newcounter{stmcounter}[section]
\newcounter{thcounter}
\newcounter{problcounter}

\numberwithin{equation}{section}

\theoremstyle{plain}
\newtheorem{cor}[stmcounter]{Corollary}
\newtheorem{stm}[stmcounter]{Statement}
\newtheorem{thm}[thcounter]{Theorem}
\newtheorem{prop}[stmcounter]{Proposition}
\newtheorem{lemma}[stmcounter]{Lemma}
\newtheorem{defin}[stmcounter]{Definition}
\newtheorem{problem}[problcounter]{Problem}

\newtheorem{claim}[stmcounter]{Claim}

\theoremstyle{definition}

\newtheorem{ex}[stmcounter]{Example}
\newtheorem{rem}[stmcounter]{Remark}
\newtheorem{con}[stmcounter]{Construction}

\begin{document}

\title[Homology of torus spaces]{Homology of torus spaces with acyclic proper faces of the orbit space}

\author{Anton Ayzenberg}
\thanks{The author is supported by the JSPS postdoctoral fellowship program}
\address{Osaka City University}
\email{ayzenberga@gmail.com}

\begin{abstract}
Let $X$ be $2n$-dimensional compact manifold and
$T^n\curvearrowright X$ be a locally standard action of a compact
torus. The orbit space $X/T$ is a manifold with corners. Suppose
that all proper faces of $X/T$ are acyclic. In the paper we study
the homological spectral sequence $\Ex^*_{*,*}\Rightarrow H_*(X)$
corresponding to the filtration of $X$ by orbit types. When the
free part of the action is not twisted, we describe the whole
spectral sequence $\Ex^*_{*,*}$ in terms of homology and
combinatorial structure of the orbit space $X/T$. In this case we
describe the kernel and the cokernel of the natural map
$\ko[X/T]/(l.s.o.p.)\to H^*(X)$, where $\ko[X/T]$ is a face ring
of $X/T$ and $(l.s.o.p.)$ is the ideal generated by a linear
system of parameters (this ideal appears as the image of
$H^{>0}(BT)$ in $H^*_T(X)$). There exists a natural double grading
on $H_*(X)$, which satisfies bigraded Poincare duality. This
general theory is applied to compute homology groups of origami
toric manifolds with acyclic proper faces of the orbit space. A
number of natural generalizations is considered. These include
Buchsbaum simplicial complexes and posets. $h'$- and $h''$-numbers
of simplicial posets appear as the ranks of certain terms in the
spectral sequence $\Ex^*_{*,*}$. In particular, using topological
argument we show that Buchsbaum posets have nonnegative
$h''$-vectors. The proofs of this paper rely on the theory of
cellular sheaves. We associate to a torus space certain sheaves
and cosheaves on the underlying simplicial poset, and observe an
interesting duality between these objects. This duality seems to
be a version of Poincare--Verdier duality between cellular sheaves
and cosheaves.
\end{abstract}

%

\maketitle

\tableofcontents

\section{Introduction}\label{SecIntro}

Let $M$ be $2n$-dimensional compact manifold with a locally
standard action of a compact torus $T^n$. This means, by
definition, that the action of $T^n$ on $M^{2n}$ is locally
modeled by a standard coordinate-wise action of $T^n$ on $\Co^n$.
Since $\Co^n/T^n$ can be identified with a nonnegative cone
$\Ro_{\geqslant}^n$, the quotient space $Q=M/T$ has a natural
structure of a compact manifold with corners. The general problem
is the following:

\begin{problem}\label{problemCohom}
Describe the (co)homology of $M$ in terms of combinatorics and
topology of the orbit space $Q$ and the local data of the action.
\end{problem}

The answer is known in the case when $Q$ and all its faces are
acyclic (so called homology polytope) \cite{MasPan}. In this case
the equivariant cohomology ring of $M$ coincides with the face
ring of simplicial poset $S_Q$ dual to $Q$, and the ordinary
cohomology has description similar to that of toric varieties or
quasitoric manifolds:
\[
H^*(M;\ko)\cong \ko[S_Q]/(\theta_1,\ldots,\theta_n),\qquad
\deg\theta_i=2.
\]
In this case $\ko[S_Q]$ is Cohen--Macaulay and
$\theta_1,\ldots,\theta_n$ is a linear regular sequence determined
by the characteristic map on $Q$. In particular, cohomology
vanishes in odd degree, and $\dim H^{2i}(M)=h_i(S_Q)$.

In general, there is a topological model of a manifold $M$, called
canonical model. The construction is the following. Start with a
nice manifold with corners $Q$, consider a principal $T^n$-bundle
$Y$ over $Q$, and then consider the quotient space $X=Y/\simc$
determined by a characteristic map \cite[def. 4.2]{Yo}. It is
known that $X$ is a manifold with locally standard torus action,
and every manifold with l.s.t.a. is equivariantly homeomorphic to
such canonical model. Thus it is sufficient to work with canonical
models to answer Problem \ref{problemCohom}.

In this paper we study the case when all proper faces of $Q$ are
acyclic, but $Q$ itself may be arbitrary. Homology of $X$ can be
described by the spectral sequence $\Ex^r_{*,*}$ associated to the
filtration of $X$ by orbit types:
\begin{equation}\label{eqIntroFiltX}
X_0\subset X_1\subset\ldots\subset X_{n-1}\subset X_n=X,\quad \dim
X_i=2i.
\end{equation}
This filtration is covered by the filtration of $Y$:
\begin{equation}\label{eqIntroFiltY}
Y_0\subset Y_1\subset\ldots\subset Y_{n-1}\subset Y_n=Y,\quad
X_i=Y_i/\simc.
\end{equation}
We prove that most entries of the second page $\Ex^2_{*,*}$
coincide with corresponding entries of $\Ey^2_{*,*}$ (Theorem
\ref{thmTwoSpecSeqGeneral}). When $Y$ is a trivial $T^n$-bundle,
$Y=Q\times T^n$, this observation allows to describe $\Ex^*_{*,*}$
completely in terms of topology and combinatorics of $Q$ (Theorem
\ref{thmEX2structure}, statement \ref{stmThmSpecConvint} and
Theorem \ref{thmBorderStruct}). This answers Problem
\ref{problemCohom} additively. From this description in particular
follows that Betti numbers of $X$ do not depend on the choice of
characteristic map. We hope, that this technic will lead to the
description of cohomology multiplication in $H^*(X)$ as well.

Another motivation for this paper comes from a theory of Buchsbaum
simplicial complexes and posets. The notions of $h'$- and
$h''$-vectors of simplicial poset $S$ first appeared in
combinatorial commutative algebra \cite{Sch,NS}. These invariants
emerge naturally in the description of homology of $X$ (Theorems
\ref{thmBorderStruct} and \ref{thmBorderManif}). The space
$X=Y/\simc$ can be constructed not only in the case when $Q$ is a
manifold with corners (``manifold case''), but also in the case
when $Q$ is a cone over geometric realization of simplicial poset
$S$ (``cone case''). In the cone case, surely, $X$ may not be a
manifold. But there still exists filtration \eqref{eqIntroFiltX},
and homology groups of $X$ can be calculated by the same method as
for manifolds, when $S$ is Buchsbaum. In the cone case we prove
that $\dim \Ex^{\infty}_{i,i}=h''(S)$ (Theorem
\ref{thmHtwoPrimes}). Thus, in particular, $h''$-vector of
Buchsbaum simplicial poset is nonnegative. This result is proved
in commutative algebra by completely different methods \cite{NS}.

The exposition of the paper is built in such way that both
situations: manifolds with acyclic faces, and cones over Buchsbaum
posets are treated in a common context. In order to do this we
introduce the notion of Buchsbaum pseudo-cell complex which is
very natural and includes both motivating examples. A theory of
cellular sheaves over simplicial posets is used to prove basic
theorems. The coincidence of most parts of $\Ex^r_{*,*}$ and
$\Ey^r_{*,*}$ follows from the Key lemma (lemma
\ref{lemmaKeyLemma}) which is an instance of general duality
between certain sheaves and cosheaves (Theorem \ref{thmDuality}).
In the manifold case this duality can be deduced from Verdier
duality for cellular sheaves, described in \cite{Curry}.

The paper is organized as follows. Section \ref{SecPrelim}
contains preliminaries on simplicial posets and cellular sheaves.
In section \ref{secHomologyOfPC} we introduce the notion of simple
pseudo-cell complex and describe spectral sequences associated to
filtrations by pseudo-cell skeleta. Section \ref{secTorusPrelim}
is devoted to torus spaces over pseudo-cell complexes. The main
results (Theorems \ref{thmTwoSpecSeqGeneral}--\ref{thmHtwoPrimes})
are stated in section \ref{secMainResults}. The rest of section
\ref{secMainResults} contains the description of homology of $X$.
There is an additional grading on homology groups, and in the
manifold case there is a bigraded Poincare duality. Section
\ref{SecKeyLemma} contains a sheaf-theoretic discussion of the
subject. In this section we prove Theorem \ref{thmDuality} which
can be considered as a version of cellular Verdier duality. This
proves the Key lemma, from which follow Theorems
\ref{thmTwoSpecSeqGeneral} and \ref{thmEX2structure}. Section
\ref{secHvectors} is devoted to the combinatorics of simplicial
posets. In this section we recall combinatorial definitions of
$f$-, $h$-, $h'$- and $h''$-vectors and prove Theorems
\ref{thmBorderStruct}--\ref{thmHtwoPrimes}. The structure of
equivariant cycles and cocycles of a manifold $X$ with locally
standard torus action is the subject of section
\ref{secEquivarGeometry}. There exists a natural map
$\ko[S]/(\theta_1,\ldots,\theta_n)\to H^*(X)$, where
$(\theta_1,\ldots,\theta_n)$ is a linear system of parameters,
associated to a characteristic map. In general (i.e. when $Q$ is
not a homology polytope) this map may be neither injective nor
surjective. The kernel of this map is described by corollary
\ref{corKernel}. The calculations for some particular examples are
gathered in section \ref{secExamples}. The main family of
nontrivial examples is the family of origami toric manifolds with
acyclic proper faces of the orbit space.

\section{Preliminary constructions}\label{SecPrelim}

\subsection{Preliminaries on simplicial posets}
First, recall several standard definitions. A partially ordered
set (poset in the following) is called simplicial if it has a
minimal element $\varnothing\in S$, and for any $I\in S$, the
lower order ideal $S_{\leqslant I}=\{J\mid J\leqslant I\}$ is
isomorphic to the boolean lattice $2^{[k]}$ (the poset of faces of
a simplex). The number $k$ is called the rank of $I\in S$ and is
denoted $|I|$. Also set $\dim I = |I|-1$. A vertex is a simplex of
rank $1$ (i.e. the atom of the poset); the set of all vertices is
denoted by $\ver(S)$. A subset $L\subset S$, for which $I<J$,
$J\in L$ implies $I\in L$ is called a simplicial subposet.

The notation $I<_iJ$ is used whenever $I<J$ and $|J|-|I|=i$. If
$S$ is a simplicial poset, then for each $I<_2J\in S_Q$ there
exist exactly two intermediate simplices $J',J''$:
\begin{equation}\label{eqSquareCond}
I<_1 J',J'' <_1 J.
\end{equation}
For simplicial poset $S$ a ``sign convention'' can be chosen. It
means that we can associate an incidence number $\inc{J}{I}=\pm 1$
to any $I<_1J\in S$ in such way that for \eqref{eqSquareCond}
holds
\begin{equation}\label{eqSignSquare}
\inc{J}{J'}\cdot \inc{J'}{I}+\inc{J}{J''}\cdot \inc{J''}{I}=0.
\end{equation}
The choice of a sign convention is equivalent to the choice of
orientations of all simplices.

For $I\in S$ consider the link:
\[
\lk_SI=\{J\in S\mid J\geqslant I\}.
\]
It is a simplicial poset with minimal element $I$. On the other
hand, $\lk_SI$ can also be considered as a subset of $S$. It can
be seen that $S\setminus\lk_SI$ is a simplicial subposet. Note,
that $\lk_S\varnothing=S$.

Let $S'$ be the barycentric subdivision of $S$. By definition,
$S'$ is a simplicial complex on the set $S\setminus \varnothing$
whose simplices are the chains of elements of $S$. By definition,
the geometric realization of $S$ is the geometric realization of
its barycentric subdivision $|S|\eqd|S'|$. One can also think
about $|S|$ as a CW-complex with simplicial cells \cite{Bj}. A
poset $S$ is called pure if all its maximal elements have equal
dimensions. A poset $S$ is pure whenever $S'$ is pure.

\begin{defin}
Simplicial complex $K$ of dimension $n-1$ is called Buchsbaum if
$\Hr_i(\lk_KI)=0$ for all $\varnothing\neq I\in K$ and $i\neq
n-1-|I|$. If $K$ is Buchsbaum and, moreover, $\Hr_i(K)=0$ for
$i\neq n-1$ then $K$ is called Cohen--Macaulay. Simplicial poset
$S$ is called Buchsbaum (Cohen--Macaulay) if $S'$ is a Buchsbaum
(resp. Cohen--Macaulay) simplicial complex.
\end{defin}

\begin{rem}
Whenever the coefficient ring in the notation of (co)homology is
omitted it is supposed to be the ground ring $\ko$, which is
either a field or the ring of integers.
\end{rem}

\begin{rem}\label{remBuchPoset}
By \cite[Sec.6]{NS}, $S$ is Buchsbaum whenever $\Hr_i(\lk_SI)=0$
for all $\varnothing\neq I\in S$ and $i\neq n-1-|I|$. Similarly,
$S$ is Cohen--Macaulay if $\Hr_i(\lk_SI)=0$ for all $I\in S$ and
$i\neq n-1-|I|$. A poset $S$ is Buchsbaum whenever all its proper
links are Cohen--Macaulay.
\end{rem}

One easily checks that Buchsbaum property implies purity.

\subsection{Cellular sheaves}
Let $\MOD_{\ko}$ be the category of $\ko$-modules. The notation
$\dim V$ is used for the rank of a $\ko$-module $V$.

Each simplicial poset $S$ defines a small category $\cat(S)$ whose
objects are the elements of $S$ and morphisms --- the inequalities
$I\leqslant J$. A cellular sheaf \cite{Curry} (or a stack
\cite{McCrory}, or a local coefficient system elsewhere) is a
covariant functor $\A\colon \cat(S)\to \MOD_{\ko}$. We simply call
$\A$ a sheaf on $S$ and hope that this will not lead to a
confusion, since different meanings of this word do not appear in
the paper. The maps $\A(J_1\leqslant J_2)$ are called the
restriction maps. The cochain complex $(\Cc^*(S;\A),d)$ is defined
as follows:
\[
\Cc^*(S;\A) = \bigoplus_{i\geqslant -1}\Cc^i(S;\A),\qquad
\Cc^i(S;\A) = \bigoplus_{\dim I=i}\A(I),
\]
\[
d\colon\Cc^i(S;\A)\to \Cc^{i+1}(S;\A), \qquad d=\bigoplus_{I<_1I',
\dim I=i}[I':I]\A(I\leqslant I').
\]
By the standard argument involving sign convention
\eqref{eqSignSquare}, $d^2=0$, thus $(\Cc^*(K;\A),d)$ is a
differential complex. Define the cohomology of $\A$ as the
cohomology of this complex:
\begin{equation}\label{eqSheafCohomDef}
H^*(S;\A)\eqd H^*(\Cc^*(S;\A),d).
\end{equation}

\begin{rem}
Cohomology of $\A$ defined this way coincide with any other
meaningful definition of cohomology. E.g. the derived functors of
the functor of global sections are isomorphic to
\eqref{eqSheafCohomDef} (see \cite{Curry} for the vast exposition
of this subject).
\end{rem}

A sheaf $\A$ on $S$ can be restricted to a simplicial subposet
$L\subset S$. The complexes $(\Cc^*(L,\A),d)$ and
$(\Cc^*(S;\A)/\Cc^*(L;\A),d)$ are defined in a usual manner. The
latter complex gives rise to a relative version of sheaf
cohomology: $H^*(S,L;\A)$.

\begin{rem}
It is standard in topological literature to consider cellular
sheaves which do not take values on $\varnothing\in S$, since in
general this element has no geometrical meaning. However, this
extra value $\A(\varnothing)$ is very important in the
considerations of this paper. Thus the cohomology group may be
nontrivial in degree $\dim\varnothing=-1$. If a sheaf $\A$ is
defined on $S$, then we often consider its truncated version
$\underline{\A}$ which coincides with $\A$ on
$S\setminus\{\varnothing\}$ and vanishes on $\varnothing$.
\end{rem}

\begin{ex}\label{exSheafConstant}
Let $W$ be a $\ko$-module. By abuse of notation let $W$ denote the
(globally) constant sheaf on $S$. It takes constant value $W$ on
$\varnothing\neq I\in S$ and vanishes on $\varnothing$; all
nontrivial restriction maps are identity isomorphisms. In this
case $H^*(S;W)\cong H^*(S)\otimes W$.
\end{ex}

\begin{ex}
A locally constant sheaf valued by $W\in\MOD_\ko$ is a sheaf
$\ca{W}$ which satisfies $\ca{W}(\varnothing)=0$, $\ca{W}(I)\cong
W$ for $I\neq\varnothing$ and all nontrivial restriction maps are
isomorphisms.
\end{ex}

\begin{ex}\label{exSheafLocal}
Let $I\in S$ and $W\in \MOD_\ko$. Consider the sheaf $\ups{I}^W$
defined by
\begin{equation}\label{eqSheafLocal}
\ups{I}^W(J)=\begin{cases} W,\mbox{ if } J\geqslant I\\0, \mbox{
otherwise},
\end{cases}
\end{equation}
with the restriction maps $\ups{I}^W(J_1\leqslant J_2)$ either
identity on $W$ (when $I\leqslant J_1$), or $0$ (otherwise). Then
$\ups{I}^W=\ups{I}^{\ko}\otimes W$ and $H^*(S;\ups{I}^W)\cong
H^*(S;\ups{I}^{\ko})\otimes W$. We have
\[
H^*(S;\ups{I}^{\ko})\cong H^{*-|I|}(\lk_SI),
\]
since corresponding differential complexes coincide.
\end{ex}

In the following if $\ca{A}$ and $\ca{B}$ are two sheaves on $S$
we denote by $\ca{A}\otimes\ca{B}$ their componentwise tensor
product: $(\ca{A}\otimes\ca{B})(I)=\ca{A}(I)\otimes\ca{B}(I)$ with
restriction maps defined in the obvious way.

\begin{ex}\label{exSheafLocTens}
As a generalization of the previous example consider the sheaf
$\ups{I}^{\ko}\otimes \ca{A}$. Then
\[
H^*(S;\ups{I}^{\ko}\otimes\ca{A})\cong
H^{*-|I|}(\lk_SI;\ca{A}|_{\lk_SI}).
\]
\end{ex}

\begin{ex}\label{exSheafLocHomol}
Following \cite{McCrory}, define $i$-th local homology sheaf
$\loc_i$ on $S$ by setting $\loc_i(\varnothing)=0$ and
\begin{equation}\label{eqDefLocHomSheaf}
\loc_i(J)=H_i(S,S\setminus\lk_SJ)
\end{equation}
for $J\neq \varnothing$. The restriction maps $\loc_i(J_1<J_2)$
are induced by inclusions $\lk_SJ_2\hookrightarrow\lk_SJ_1$. A
poset $S$ is Buchsbaum if and only if $\loc_i=0$ for $i<n-1$.
\end{ex}

\begin{defin}
Buchsbaum poset $S$ is called homology manifold (orientable over
$\ko$) if its local homology sheaf $\loc_{n-1}$ is isomorphic to
the constant sheaf $\ko$.
\end{defin}

If $|S|$ is a compact closed orientable topological manifold then
$S$ is a homology manifold.

\subsection{Cosheaves}

A cellular cosheaf (see \cite{Curry}) is a contravariant functor
$\cah{A}\colon \catop(S)\to \MOD_{\ko}$. The homology of a cosheaf
is defined similar to cohomology of sheaves:
\begin{gather*}
\Cc_*(S;\cah{A}) = \bigoplus_{i\geqslant -1}\Cc_i(S;\cah{A})\quad
\Cc_i(S;\cah{A}) = \bigoplus_{\dim I=i}\A(I)\\
d\colon\Cc_i(S;\cah{A})\to \Cc_{i-1}(S;\cah{A}),\quad
d=\bigoplus_{I>_1I', \dim I=i}[I:I']\cah{A}(I\geqslant I'),\\
H_*(S;\cah{A})\eqd H_*(\Cc_*(S;\cah{A}),d).
\end{gather*}

\begin{ex}
Each locally constant sheaf $\ca{W}$ on $S$ defines the locally
constant cosheaf $\cah{W}$ by inverting arrows, i.e.
$\cah{W}(I)\cong \ca{W}(I)$ and $\cah{W}(I>J)=(\ca{W}(J<I))^{-1}$.
\end{ex}

\section{Buchsbaum pseudo-cell
complexes}\label{secHomologyOfPC}

\subsection{Simple pseudo-cell complexes}
\begin{defin}[Pseudo-cell complex]
A CW-pair $(F,\dd F)$ will be called $k$-dimensional pseudo-cell,
if $F$ is compact and connected, $\dim F = k$, $\dim \dd
F\leqslant k-1$. A (regular finite) pseudo-cell complex $Q$ is a
space which is a union of an expanding sequence of subspaces $Q_k$
such that $Q_{-1}$ is empty and $Q_{k}$ is the pushout obtained
from $Q_{k-1}$ by attaching finite number of $k$-dimensional
pseudo-cells $(F,\dd F)$ along injective attaching maps $\dd F\to
Q_{k-1}$. The images of $(F,\dd F)$ in $Q$ will be also called
pseudo-cells and denoted by the same letters.
\end{defin}

\begin{rem}
In general, situations when $\dd F=\varnothing$ or
$Q_0=\varnothing$ are allowed by this definition. Thus the
construction of pseudo-cell complex may actually start not from
$Q_0$ but from higher dimensions.
\end{rem}

Let $F\oo = F\setminus \dd F$ denote open cells. In the following
we assume that the boundary of each cell is a union of lower
dimensional cells. Thus all pseudo-cells of $Q$ are partially
ordered by inclusion. We denote by $S_Q$ the poset of faces with
the reversed order, i.e. $F<_{S_Q} G$ iff $G\subseteq \dd F\subset
F$. To distinguish abstract elements of poset $S_Q$ from faces of
$Q$ the former are denoted by $I, J,\ldots\in S_Q$, and
corresponding faces --- $F_I,F_J,\ldots\subseteq Q$.

\begin{defin}\label{definSimplePseudocellCpx}
A pseudo-cell complex $Q$, $\dim Q=n$ is called simple if $S_Q$ is
a simplicial poset of dimension $n-1$ and $\dim F_I=n-1-\dim I$
for all $I\in S_Q$.
\end{defin}

Thus for every face $F$, the upper interval $\{G\mid G\supseteq
F\}$ is isomorphic to a boolean lattice $2^{[\codim F]}$. In
particular, there exists a unique maximal pseudo-cell
$F_{\varnothing}$ of dimension $n$, i.e. $Q$ itself. In case of
simple pseudo-cell complexes we adopt the following naming
convention: pseudo-cells different from $F_{\varnothing}=Q$ are
called faces, and faces of codimension $1$
--- facets. Facets correspond to vertices (atoms) of $S_Q$. Each
face $F$ is contained in exactly $\codim F$ facets. In this paper
only simple pseudo-cell complexes are considered.

\begin{ex}\label{exSimplePseudoManif} Nice (compact connected) manifolds
with corners as defined in \cite{MasPan} are examples of simple
pseudo-cell complexes. Each face $F$ is itself a manifold and $\dd
F$ is the boundary in a common sense.
\end{ex}

\begin{ex}\label{exSimplePseudoPoset} Each pure simplicial poset $S$ determines a simple
pseudo-cell complex $P(S)$ such that $S_{P(S)}= S$ by the
following standard construction. Consider the barycentric
subdivision $S'$ and construct the cone $P(S)=|\cone S'|$. By
definition, $\cone S'$ is a simplicial complex on the set $S$ and
$k$-simplices of $\cone S'$ have the form $(I_0< I_1<\ldots<I_k)$,
where $I_i\in S$. For each $I\in S$ consider the pseudo-cell:
\[
F_I=\left|\{(I_0< I_1<\ldots)\in \cone S'\mbox{ such that }
I_0\geqslant I\}\right|\subset |\cone S'|
\]
\[
\dd F_I=\left|\{(I_0< I_1<\ldots)\in \cone S'\mbox{ such that }
I_0>I\}\right|\subset |\cone S'|
\]
Since $S$ is pure, $\dim F_I=n-\dim I-1$. These sets define a
pseudo-cell structure on $P(S)$. One shows that $F_I\subset F_J$
whenever $J<I$. Thus $S_{P(S)}=S$. Face $F_I$ is called dual to
$I\in S$. The filtration by pseudo-cell skeleta
\begin{equation}\label{eqCoskeleton}
\varnothing=Q_{-1}\subset Q_0\subset Q_1\subset\ldots\subset
Q_{n-1}=\dd Q = |S|,
\end{equation}
is called the coskeleton filtration of $|S|$ (see \cite{McCrory}).

The maximal pseudo-cell $F_{\varnothing}$ of $P(S)$ is
$P(S)\cong\cone|S|$, and $\dd F_{\varnothing}=|S|$. Note that $\dd
F_I$ can be identified with the barycentric subdivision of
$\lk_SI$. Face $F_I$ is the cone over $\dd F_I$.

If $S$ is non-pure, this construction makes sense as well, but the
dimension of $F_I$ may not be equal to $n-\dim I-1$. So $P(S)$ is
not a simple pseudo-cell complex if $S$ is not pure.
\end{ex}

For a general pseudo-cell complex $Q$ there is a skeleton
filtration
\begin{equation}\label{eqFiltrationQ}
Q_0\subset Q_1\subset \ldots\subset Q_{n-1}=\dd Q\subset Q_n=Q
\end{equation}
and the corresponding spectral sequences in homology and
cohomology are:
\begin{gather}
\Eq^1_{p,q}=H_{p+q}(Q_p,Q_{p-1}) \Rightarrow H_{p+q}(Q),\qquad
\dqq^r\colon \Eq^r_{*,*}\to\Eq^r_{*-r,*+r-1}\\
\Eq_1^{p,q}=H^{p+q}(Q_p,Q_{p-1}) \Rightarrow H^{p+q}(Q)\qquad
(\dqq)_r\colon \Eq_r^{*,*}\to\Eq_r^{*+r,*-r+1}.
\end{gather}
In the following only homological case is considered; the
cohomological case being completely parallel.

Similar to ordinary cell complexes the first term of the spectral
sequence is described as a sum:
\[
H_{p+q}(Q_p,Q_{p-1})\cong \bigoplus_{\dim F=p}H_{p+q}(F,\dd F).
\]
The differential $\dqq^1$ is the sum over all pairs $I<_1J\in S$
of the maps:
\begin{multline}\label{eqMattachingMap}
m^q_{I,J}\colon H_{q+\dim F_I}(F_I,\dd F_I)\to H_{q+\dim
F_I-1}(\dd F_I)\to \\ \to H_{q+\dim F_I-1}(\dd F_I, \dd
F_I\setminus F_J\oo)\cong H_{q+\dim F_J}(F_J, \dd F_J),
\end{multline}
where the last isomorphism is due to excision. Also consider the
truncated spectral sequence
\[
\Eqt^1_{p,q}=H_{p+q}(Q_p,Q_{p-1}), p<n\Rightarrow H_{p+q}(\dd Q).
\]

\begin{con}\label{conSheafOnQ}
Given a sign convention on $S_Q$, for each $q$ consider the sheaf
$\hh_q$ on $S_Q$ given by
\[
\hh_q(I)=H_{q+\dim F_I}(F_I,\dd F_I)
\]
with restriction maps $\hh_q(I<_1J) = \inc{J}{I} m^q_{I,J}$. For
general $I<_kJ$ consider any saturated chain
\begin{equation}\label{eqSaturChain}
I<_1J_1<_1\ldots<_1J_{k-1}<_1J
\end{equation}
and set $\hh_q(I<_kJ)$ to be equal to the composition
\[
\hh_q(J_{k-1}<_1J)\circ\ldots\circ\hh_q(I<_1J_1).
\]

\begin{lemma}
The map $\hh_q(I<_kJ)$ does not depend on a saturated chain
\eqref{eqSaturChain}.
\end{lemma}

\begin{proof}
The differential $\dqq^1$ satisfies $(\dqq^1)^2=0$, thus
$m^q_{J',J}\circ m^q_{I,J'}+m^q_{J'',J}\circ m^q_{I,J''}=0$. By
combining this with \eqref{eqSignSquare} we prove that
$\hh_q(I<_2J)$ is independent of a chain. In general, since
$\{T\mid I\leqslant T\leqslant J\}$ is a boolean lattice, any two
saturated chains between $I$ and $J$ are connected by a sequence
of elementary flips $[J_k~<_1~T_1~<_1~J_{k+2}] \rightsquigarrow
[J_k~<_1~T_2~<_1~J_{k+2}]$.
\end{proof}

Thus the sheaves $\hh_q$ are well defined. These sheaves will be
called the structure sheaves of $Q$. Consider also the truncated
structure sheaves
\[
\hht_q(I)=\begin{cases} \hh_q(I) \mbox{ if } I\neq\varnothing,\\
0,\mbox{ if } I=\varnothing.
\end{cases}
\]
\end{con}

\begin{cor}
The cochain complexes of structure sheaves coincide with
$\Eq^1_{*,*}$ up to change of indices:
\[
(\Eq^1_{*,q},\dqq^1)\cong(C^{n-1-*}(\hh_q),d),\qquad
(\Eqt^1_{*,q},\dqq^1)\cong(C^{n-1-*}(\hht_q),d).
\]
\end{cor}

\begin{proof}
Follows from the definition of the cochain complex of a sheaf.
\end{proof}

\begin{rem}\label{remLocalSheavesOnPoset}
Let $S$ be a pure simplicial poset of dimension $n-1$ and $P(S)$
--- its dual simple pseudo-cell complex. In this case there exists an
isomorphism of sheaves
\begin{equation}\label{eqLocStrIsom}
\hht_q\cong \loc_{q+n-1},
\end{equation}
where $\loc_*$ are the sheaves of local homology defined in
example \ref{exSheafLocHomol}. Indeed, it can be shown that
$H_i(S,S\setminus\lk_SI)\cong H_{i-\dim I}(F_I,\dd F_I)$ and these
isomorphisms can be chosen compatible with restriction maps. For
simplicial complexes this fact is proved in
\cite[Sec.6.1]{McCrory}; the case of simplicial posets is rather
similar. Note that $\hh_q$ depends on the sign convention while
$\loc$ does not. There is a simple explanation: the isomorphism
\eqref{eqLocStrIsom} itself depends on the orientations of
simplices.
\end{rem}

\subsection{Buchsbaum pseudo-cell complexes}
\begin{defin}\label{defBuchsCell}
A simple pseudo-cell complex $Q$ of dimension $n$ is called
Buchsbaum if for any face $F_I\subset Q$, $I\neq\varnothing$ the
following conditions hold:
\begin{enumerate}
\item $F_I$ is acyclic over $\Zo$;
\item $H_i(F_I,\dd F_I)=0$ if $i\neq\dim F_I$.
\end{enumerate}
Buchsbaum complex $Q$ is called Cohen--Macaulay if these two
conditions also hold for $I=\varnothing$.
\end{defin}

The second condition in Buchsbaum case is equivalent to $\hht_q =
0$ for $q\neq 0$. Cohen--Macaulay case is equivalent to $\hh_q =
0$ for $q\neq 0$. Obviously, $Q$ is Buchsbaum if and only if all
its proper faces are Cohen--Macaulay. Thus any face of dimension
$p\geqslant 1$ has nonempty boundary of dimension $p-1$. In
particular, this implies $S_Q$ is pure.

\begin{defin}
Buchsbaum pseudo-cell complex $Q$ is called ($\ko$-orientable)
Buchsbaum manifold if $\hht_0$ is isomorphic to a constant sheaf
$\ko$.
\end{defin}

Note that this definition actually describes only the property of
$\dd Q$ not $Q$ itself.

\begin{ex}
If $Q$ is a nice compact manifold with corners in which every
proper face is acyclic and orientable, then $Q$ is a Buchsbaum
pseudo-cell complex. Indeed, the second condition of
\ref{defBuchsCell} follows by Poincare--Lefschetz duality. If,
moreover, $Q$ is orientable itself then $Q$ is a Buchsbaum
manifold (over all $\ko$). Indeed, the restriction maps
$\hh_0(\varnothing\subset I)$ send the fundamental cycle $[Q]\in
H_n(Q,\dd Q)\cong\ko$ to fundamental cycles of proper faces, thus
identifying $\hht_0$ with the constant sheaf $\ko$. The choice of
orientations establishing this identification is described in
details in section \ref{secEquivarGeometry}.
\end{ex}

\begin{ex}
Simplicial poset $S$ is Buchsbaum (resp. Cohen--Macaulay) whenever
$P(S)$ is a Buchsbaum (resp. Cohen--Macaulay) simple pseudo-cell
complex. Indeed, any face of $P(S)$ is a cone, thus contractible.
On the other hand, $H_i(F_I,\dd F_I)\cong
H_i(\cone|\lk_SI|,|\lk_SI|)\cong \Hr_{i-1}(|\lk_SI|)$. Thus
condition 2 in definition \ref{defBuchsCell} is satisfied whenever
$\Hr_{i}(\lk_SI)=0$ for $i\neq n-1-|I|$. This is equivalent to
Buchsbaumness (resp. Cohen--Macaulayness) of $S$ by remark
\ref{remBuchPoset}.

Poset $S$ is a homology manifold if and only if $P(S)$ is a
Buchsbaum manifold. This follows from remark
\ref{remLocalSheavesOnPoset}. In particular, if $|S|$ is a closed
orientable manifold then $P(S)$ is a Buchsbaum manifold.
\end{ex}

In general, if $Q$ is Buchsbaum, then its underlying poset $S_Q$
is also Buchsbaum, see lemma \ref{lemmaUniverPosets} below.

In Buchsbaum (resp. Cohen--Macaulay) case the spectral sequence
$\Eqt$ (resp. $\Eq$) collapses at the second page, thus
\[
H^{n-1-p}(S_Q;\hht_0)\cong
\Eqt^2_{p,0}\stackrel{\cong}{\Rightarrow} H_p(\dd Q),\qquad\mbox{
if } Q\mbox{ is Buchsbaum}
\]
\[
H^{n-1-p}(S_Q;\hh_0)\cong \Eq^2_{p,0}\stackrel{\cong}{\Rightarrow}
H_p(Q),\qquad\mbox{ if }  Q\mbox{ is Cohen--Macaulay}
\]
In particular, if $Q$ is a Buchsbaum manifold, then
\begin{equation}\label{eqBuchManPoincS}
H^{n-1-p}(S_Q)\cong H_p(\dd Q)
\end{equation}

Let $Q$ be a simple pseudo-cell complex, and $\varnothing \neq
I\in S_Q$. The face $F_I$ is a simple pseudo-cell complex itself,
and $S_{F_I}=\lk_SI$. The structure sheaves of $F_I$ are the
restrictions of $\hh_q$ to $\lk_{S_Q}I\subset S_Q$. If $Q$ is
Buchsbaum, then $F_I$ is Cohen--Macaulay, thus
\begin{equation}\label{eqCollapsingAtCMFace}
H^k(\lk_SI;\hh_0)\stackrel{\cong}{\Rightarrow} H_{\dim
F_I-1-k}(F_I),
\end{equation}
which is either $\ko$ (in case $k=\dim F_I-1$) or $0$ (otherwise),
since $F_I$ is acyclic.

\subsection{Universality of posets}
The aim of this subsection is to show that Buchsbaum pseudo-cell
complex coincides up to homology with the underlying simplicial
poset away from maximal cells. This was proved for nice manifolds
with corners in \cite{MasPan} and essentially we follow the proof
given there.

\begin{lemma}\label{lemmaUniverPosets}\mbox{}

$(1)_n$ Let $Q$ be Buchsbaum pseudo-cell complex of dimension $n$,
$S_Q$ --- its underlying poset, and $P=P(S_Q)$ --- simple
pseudo-cell complex associated to $S_Q$ (example
\ref{exSimplePseudoPoset}), $\dd P=|S_Q|$. Then there exists a
face-preserving map $\varphi\colon Q\to P$ which induces the
identity isomorphism of posets and the isomorphism of the
truncated spectral sequences $\varphi_*\colon
\Eqt^r_{*,*}\stackrel{\cong}{\to}\Ept^r_{*,*}$ for $r\geqslant 1$.

$(2)_n$ If $Q$ is Cohen--Macaulay of dimension $n$, then $\varphi$
induces the isomorphism of non-truncated spectral sequences
$\varphi_*\colon \Eq^r_{*,*}\stackrel{\cong}{\to}\Ep^r_{*,*}$.
\end{lemma}

\begin{proof}
The map $\varphi$ is constructed inductively. $0$-skeleta of $Q$
and $P$ are naturally identified. There always exists an extension
of $\varphi$ to higher-dimensional faces since all pseudo-cells of
$P$ are cones. The lemma is proved by the following scheme of
induction: $(2)_{\leqslant n-1}\Rightarrow (1)_n\Rightarrow
(2)_n$. The case $n=0$ is clear. Let us prove $(1)_n\Rightarrow
(2)_n$. The map $\varphi$ induces the homomorphism of the long
exact sequences:
\[
\xymatrix{ \Hr_*(\dd
Q)\ar@{->}[r]\ar@{->}[d]&\Hr_*(Q)\ar@{->}[r]\ar@{->}[d]& H_*(Q,\dd
Q)\ar@{->}[r]\ar@{->}[d]&\Hr_{*-1}(\dd
Q)\ar@{->}[r]\ar@{->}[d]&\Hr_{*-1}(Q)\ar@{->}[d]\\
\Hr_*(\dd P)\ar@{->}[r]&\Hr_*(P)\ar@{->}[r]& H_*(P,\dd
P)\ar@{->}[r]&\Hr_{*-1}(\dd P)\ar@{->}[r]&\Hr_{*-1}(P) }
\]
The maps $\Hr_*(Q)\to\Hr_*(P)$ are isomorphisms since both groups
are trivial. The maps $\Hr_*(\dd Q)\to\Hr_*(\dd P)$ are
isomorphisms by $(1)_n$, since
$\Eqt\stackrel{\cong}{\Rightarrow}H_*(\dd Q)$ and
$\Ept\stackrel{\cong}{\Rightarrow}H_*(\dd P)$. Five lemma shows
that $\varphi_*\colon\Eq^1_{n,*}\to \Ep^1_{n,*}$ is an isomorphism
as well. This imply $(2)_n$.

Now we prove $(2)_{\leqslant n-1}\Rightarrow (1)_n$. Let $F_I$ be
faces of $Q$ and $\widetilde{F}_I$ --- faces of $P$. All proper
faces of $Q$ are Cohen--Macaulay of dimension $\leqslant n-1$.
Thus $(2)_{\leqslant n-1}$ implies isomorphisms $H_*(F_I,\dd
F_I)\to H_*(\widetilde{F}_I,\dd \widetilde{F}_I)$ which sum
together to the isomorphism
$\varphi_*\colon\Eqt^1_{*,*}\stackrel{\cong}{\to}\Ept^1_{*,*}$.
\end{proof}

\begin{cor}\label{corQbuchSbuch}
If $Q$ is a Buchsbaum (resp. Cohen--Macaulay) pseudo-cell complex,
then $S_Q$ is a Buchsbaum (resp. Cohen--Macaulay) simplicial
poset. If $Q$ is a Buchsbaum manifold, then $S_Q$ is a homology
manifold.
\end{cor}

In particular, according to lemma \ref{lemmaUniverPosets}, if $Q$
is Buchsbaum, then $\dd Q$ is homologous to $|S_Q|=\dd P(S_Q)$. So
in the following we may not distinguish between their homology. If
$Q$ is Buchsbaum manifold, then \eqref{eqBuchManPoincS} implies
Poincare duality for $\dd Q$:
\begin{equation}\label{eqBuchManPoinc}
H^{n-1-p}(\dd Q)\cong H_p(\dd Q).
\end{equation}

\subsection{Structure of $\Eq^r_{*,*}$ in Buchsbaum
case}\label{subsecSpecSeqQ} Let $\delta_i\colon H_i(Q,\dd Q)\to
H_{i-1}(\dd Q)$ be the connecting homomorphisms in the long exact
sequence of the pair $(Q,\dd Q)$.

\begin{lemma}
The second term of $\Eq^*_{*,*}$ for Buchsbaum pseudo-cell complex
$Q$ is described as follows:
\begin{equation}\label{eqQsseq2}
\Eq^2_{p,q}\cong\begin{cases} H_p(\dd Q), \mbox{ if } p\leqslant n-2, q=0,\\
\coker\delta_n,\mbox{ if }p=n-1, q=0,\\
\Ker\delta_n,\mbox{ if }p=n, q=0,\\
H_{n+q}(Q,\dd Q),\mbox{ if } p=n, q<0,\\
0,\mbox{ otherwise}.
\end{cases}
\end{equation}
\end{lemma}

\begin{proof}
The first page of the non-truncated spectral sequence has the form
\begin{tikzpicture}
\tikzstyle{column 6}=[anchor=base west]

\matrix (m) [matrix of math nodes, nodes in empty
cells,nodes={minimum width=5ex, minimum height=5ex,outer
sep=-5pt}, column sep=0ex,row sep=0ex]{
                &     &     &   &   &  & \\
                &  C^{n-1}(S;\hh_0) &  \ldots  & C^1(S;\hh_0)  & C^0(S;\hh_0) & C^{-1}(S;\hh_0)=H_n(Q,\dd Q) & \\
    \strut      & 0  & \ldots &  0  &  0 &  C^{-1}(S;\hh_{-1})=H_{n-1}(Q,\dd Q) &\strut\\
                & \vdots  &    &  \vdots  & \vdots   &  \vdots &\\
                &  0 &  \ldots  & 0  & 0  &  C^{-1}(S;\hh_{-n})=H_0(Q,\dd Q) &\\};
\draw[thick] (m-1-1.east) -- ([yshift=-2mm]m-5-1.east);

\draw[thick] ([yshift=+3mm]m-2-1.north) --
([yshift=+3mm]m-2-7.north);

\draw ([xshift=+1mm,yshift=-1mm]m-5-1.south east) node{$q$};

\draw ([xshift=+3mm,yshift=+3mm]m-2-7.north) node{$p$};

\draw ([xshift=-6mm,yshift=+2mm]m-1-1.east) node{$\Eq^1_{p,q}$};

\draw ([yshift=+4mm]m-2-2.north) node{\small $0$}; \draw
([yshift=+4mm]m-2-4.north) node{\small $n-2$}; \draw
([yshift=+4mm]m-2-5.north) node{\small $n-1$}; \draw
([yshift=+4mm]m-2-6.north) node{\small $n$};

\draw ([xshift=-3mm]m-2-1.east) node{\small $0$}; \draw
([xshift=-3mm]m-3-1.east) node{\small $-1$}; \draw
([xshift=-3mm,yshift=+1mm]m-5-1.east) node{\small $-n$};

\end{tikzpicture}

By the definition of Buchsbaum complex, $\Eq^2_{p,q}=0$ if $p<n$
and $q\neq 0$. Terms of the second page with $p\leqslant n-2$
coincide with their non-truncated versions:
$\Eq^2_{p,0}=\Eqt^2_{p,0}\cong H^{n-1-p}(S_Q;\hht_0)\cong H_p(\dd
Q)$. For $p=n$, $q<0$ the first differential vanishes, thus
$\Eq^2_{n,q}=\Eq^1_{n,q}\cong H_{n+q}(Q,\dd Q)$. The only two
cases that require further investigation are $(p,q)=(n-1,0)$ and
$(n,0)$. To describe these cases consider the short exact sequence
of sheaves
\begin{equation}\label{eqTruncShort}
0\rightarrow \hht_0\rightarrow \hh_0\rightarrow
\hh_0/\hht_0\rightarrow 0.
\end{equation}
The quotient sheaf $\hh_0/\hht_0$ is concentrated in degree $-1$
and its value on $\varnothing$ is $H_n(Q,\dd Q)$. Sequence
\eqref{eqTruncShort} induces the long exact sequence in cohomology
(middle row):
\[
\xymatrix{
&& H_n(Q,\dd Q)\ar@{->}[r]^{\delta_n}&H_{n-1}(\dd Q)&&\\
0\ar@{->}[r]& H^{-1}(S_Q;\hh_0)\ar@{->}[r]\ar@{->}[d]^{\cong}&
H^{-1}(S_Q;\hh_0/\hht_0)\ar@{->}[r]\ar@{->}[u]^{\cong}&
H^0(S_Q;\hht_0)\ar@{->}[r]\ar@{->}[u]^{\cong}\ar@{->}[d]^{\cong}&
H^0(S_Q;\hh_0)\ar@{->}[r]\ar@{->}[d]^{\cong}&0\\
&\Eq^2_{n,0}&&\Eqt^2_{n-1,0}\ar@{->}[r]&\Eq^2_{n-1,0} }
\]
Thus $\Eq^2_{n,0}\cong \Ker\delta_n$ and $\Eq^2_{n-1,0}\cong
\coker\delta_n$.
\end{proof}

In the situations like this, we call a spectral sequence
\ang-shaped. The only non-vanishing differentials in $\Eq^r$ for
$r\geqslant 2$ are $d_r\colon \Eq^r_{n,1-r}\to \Eq^r_{n-r,0}$.
They have pairwise different domains and targets, thus
$\Eq^r_{p,q}\Rightarrow H_{p+q}(Q)$ folds in a long exact
sequence, which is isomorphic to the long exact sequence of the
pair $(Q,\dd Q)$:

\begin{equation}
\xymatrix{\ldots\ar@{->}[r]& H_i(Q)\ar@{=}[dd]\ar@{->}[r]&
\Eq_{n,i-n}^{n+1-i}\ar@{->}[r]^{\dqq^{n+1-i}}&
\Eq_{i-1,0}^{n+1-i}\ar@{->}[r]&
H_{i-1}(Q)\ar@{->}[r]\ar@{=}[dd]&\ldots\\
&&\Eq^1_{n,i-n}\ar@{->}[u]_{\cong}\ar@{=}[d]&\Eq^2_{i-1,0}\ar@{->}[d]^{\cong}\ar@{->}[u]_{\cong}&&\\
\ldots\ar@{->}[r]& H_i(Q)\ar@{->}[r]&H_i(Q,\dd Q)
\ar@{->}[r]^{\delta_i}& H_{i-1}(\dd
Q)\ar@{->}[r]&H_{i-1}(Q)\ar@{->}[r]&\ldots}
\end{equation}

This gives a complete characterization of $\Eq$ in terms of the
homological long exact sequence of the pair $(Q,\dd Q)$.

\subsection{Artificial page $\Eq^{1+}_{*,*}$}\label{subsecExtraTermQ}
In this subsection we formally introduce an additional term in the
spectral sequence to make description of $\Eq^*_{*,*}$ more
convenient and uniform. The goal is to carry away $\delta_n$
(which appears in \eqref{eqQsseq2}) from the description of the
page and treat it as one of higher differentials.

Let $\Ex^{1+}_{*,*}$ be the collection of $\ko$-modules defined by
\[
\Eq^{1+}_{p,q}\eqd\begin{cases} \Eqt^2_{p,q}, \mbox{ if } p\leqslant n-1,\\
\Eq^1_{p,q},\mbox{ if } p=n,\\
0,\mbox{ otherwise}.
\end{cases}
\]
Let $\dqq^{1-}$ be the differential of degree $(-1,0)$ operating
on $\bigoplus\Eq^1_{p,q}$ by:
\[
\dqq^{1-}=\begin{cases}\dqq^1\colon \Eq^1_{p,q}\to\Eq^1_{p-1,q},
\mbox{ if } p\leqslant n-1,\\0, \mbox{ otherwise}
\end{cases}
\]
It is easily seen that $H(\Eq^1_{*,*};\dqq^{1-})$ is isomorphic to
$\Eq^{1+}_{*,*}$. Now consider the differential $\dqq^{1+}$ of
degree $(-1,0)$ operating on $\bigoplus\Eq^{1+}_{p,q}$:
\[
\dqq^{1+}=\begin{cases} 0, \mbox{ if }p\leqslant n-1;\\
\Eq^1_{n,q}\stackrel{\dqq^1}{\longrightarrow}\Eq^1_{n-1,q}
\longrightarrow \Eq^2_{n-1,q},\mbox{ if }p=n.
\end{cases}
\]
Then $\Eq^2\cong H(\Eq^{1+},\dqq^{1+})$. These considerations are
shown on the diagram:
\[
\xymatrix{
&\Eq^{1+}\ar@{..>}[rd]^{\dqq^{1+}}&&&\\
\Eq^1\ar@{..>}[rr]^{\dqq^1}\ar@{..>}[ru]^{\dqq^{1-}}&&
\Eq^2\ar@{..>}[r]^{\dqq^2}&\Eq^3\ar@{..>}[r]^{\dqq^3}&\ldots }
\]
in which the dotted arrows represent passing to homology. To
summarize:

\begin{claim}\label{claimFormalPage}
There is a spectral sequence whose first page is $(\Eq^{1+},
\dqq^{1+})$ and subsequent terms coincide with $\Eq^r$ for
$r\geqslant 2$. Its nontrivial differentials for $r\geqslant 1$
are the maps
\[
\dqq^r\colon \Eq^r_{n,1-r}\to \Eq^r_{n-r,0}
\]
which coincide up to isomorphism with
\[
\delta_{n+1-r}\colon
H_{n+1-r}(Q,\dd Q)\to H_{n-r}(\dd Q).
\]
\end{claim}

Thus the spectral sequence $\Eq^r_{*,*}$ for $r\geqslant 1+$ up to
isomorphism has the form
\begin{tikzpicture}
\tikzstyle{column 6}=[anchor=base west]

\matrix (m) [matrix of math nodes, nodes in empty
cells,nodes={minimum width=5ex, minimum height=5ex,outer
sep=-5pt}, column sep=2ex,row sep=0ex]{
                &      &     &     &     &[10mm]     & \\
                &  H_0(\dd Q) &  \ldots  & H_{n-2}(\dd Q)  & H_{n-1}(\dd Q) & H_n(Q,\dd Q) & \\
    \strut      &   &  &    &   &  H_{n-1}(Q,\dd Q) &\strut\\
                &   &    &   &   &  \vdots &\\
                &  &    &   &   &  H_1(Q,\dd Q) &\\};
\draw[thick] (m-1-1.east) -- (m-5-1.east);

\draw[thick] ([yshift=+5mm]m-2-1.north) --
([xshift=-4mm,yshift=+5mm]m-2-7.north);

%

\draw ([xshift=-6mm,yshift=+2mm]m-1-1.east)
node{$\Eq^{1+}_{p,q}$};

\tikzstyle{every node}=[midway,auto,font=\scriptsize]

\draw[-stealth] ([xshift=-1mm]m-2-6.west) --
node[above]{$\dqq^{1+}=\delta_n$} ([xshift=+1mm]m-2-5.east);

\draw[-stealth] ([xshift=-2mm]m-3-6.west) --
node{$\dqq^{2}=\delta_{n-1}$} (m-2-4.south east);

\draw[-stealth] ([xshift=-2mm]m-5-6.west) --
node{$\dqq^{n}=\delta_1$} (m-2-2.south east);

\end{tikzpicture}


\section{Torus spaces over Buchsbaum pseudo-cell
complexes}\label{secTorusPrelim}

\subsection{Preliminaries on torus maps}\label{subsecT}
Let $\ld$ be a nonnegative integer. Consider a compact torus
$T^\ld=(S^1)^\ld$. The homology algebra $H_*(T^\ld;\ko)$ is the
exterior algebra $\Lambda = \Lambda_{\ko}[H_1(T^\ld)]$. Let
$\Lambda^{(q)}$ denote the graded component of $\Lambda$ of degree
$q$, $\Lambda=\bigoplus_{q=0}^\ld\Lambda^{(q)}$,
$\Lambda^{(q)}\cong H_q(T^\ld)$, $\dim \Lambda^{(q)}={\ld\choose
q}$.

If $T^\ld$ acts on a space $Z$, then $H_*(Z)$ obtains the
structure of $\Lambda$-module (i.e. two-sided $\Lambda$-module
with property $a\cdot x=(-1)^{\deg a\deg x}x\cdot a$ for $a\in
\Lambda$, $x\in H_*(Z)$); $T^\ld$-equivariant maps $f\colon Z_1\to
Z_2$ induce module homomorphisms $f_*\colon H_*(Z_1)\to H_*(Z_2)$;
and equivariant filtrations induce spectral sequences with
$\Lambda$-module structures.

\begin{con}
Let $\Tt_\ld$ be the set of all 1-dimensional toric subgroups of
$T^\ld$. Let $M$ be a finite subset of $\Tt_\ld$, i.e. a
collection of subgroups $M=\{T_s^1,i_s\colon T_s^1\hookrightarrow
T^\ld\}$. Consider the homomorphism
\[
i_M\colon T^M\to T^\ld, \quad T^M\eqd\prod\nolimits_MT_s^1,\quad
i_M\eqd\prod\nolimits_Mi_s.
\]

\begin{defin}
We say that the collection $M$ of $1$-dimensional subgroups
satisfies $\sta{\ko}$-condition if the map $(i_M)_*\colon
H_1(T^M;\ko)\to H_1(T^\ld;\ko)$ is injective and splits.
\end{defin}

If $i_M$ itself is injective, then $M$ satisfies $\sta{\ko}$ for
$\ko=\Zo$ and all fields. Moreover, $\sta{\Zo}$ is equivalent to
injectivity of $i_M$. Generally, $\sta{\ko}$ implies that
$\Gamma=\ker i_M$ is a finite subgroup of $T^M$.

For a set $M$ satisfying $\sta{\ko}$ consider the exact sequence
\[
0\rightarrow \Gamma \rightarrow T^M
\stackrel{i_M}{\longrightarrow} T^\ld
\stackrel{\rho}{\longrightarrow} G\rightarrow 0,
\]
where $G=T^\ld/i_M(T^M)$ is isomorphic to a torus $T^{\ld-|M|}$.

\begin{lemma}\label{lemmaIdealQuotTorus}
Let $\I_M$ be the ideal of $\Lambda$ generated by $i_M(H_1(T^M))$.
Then there exists a unique map $\beta$ which encloses the diagram
\[
\xymatrix{H_*(T^\ld)\ar@{->}[r]^{\rho_*}\ar@{->}[d]^{\cong}&H_*(G)\\
\Lambda \ar@{->}[r]^{q}& \Lambda/\I_M\ar@{..>}[u]_{\beta}}
\]
and $\beta$ is an isomorphism.
\end{lemma}

\begin{proof}
We have $\rho_*\colon H_*(T^\ld)\to H_*(G)\cong
\Lambda^*[H_1(G)]$. Map $\rho$ is $T^\ld$-equivariant, thus
$\rho_*$ is a map of $\Lambda$-modules. Since
$\rho_*((i_M)_*H_1(T^M))=0$, we have $\rho_*(\I_M)=0$, thus
$\rho_*$ factors through the quotient module, $\rho_*=\beta\circ
q$. Since $\rho_*$ is surjective so is $\beta$. By
$\sta{\ko}$-condition we have a split exact sequence
\[
0\rightarrow H_1(T^M)\rightarrow H_1(T^\ld)\rightarrow
H_1(G)\rightarrow 0,
\]
So far there is a section $\alpha\colon H_1(G)\to H_1(T^\ld)$ of
the map $\rho_*$ in degree $1$. This section extends to
$\widetilde{\alpha}\colon H_*(G)=\Lambda^*[H_1(G)]\to \Lambda$,
which is a section of $\rho_*$. Thus $\beta$ is injective.
\end{proof}
\end{con}

\subsection{Principal torus bundles}

Let $\rho\colon Y\to Q$ be a principal $T^\ld$-bundle over a
simple pseudo-cell complex $Q$.

\begin{lemma}\label{lemmaCMtrivialis}
If $Q$ is Cohen--Macaulay, then $Y$ is trivial. More precisely,
there exists an isomorphism $\xi$:
\[
\xymatrix{Y\ar@{->}[rr]^{\xi}\ar@{->}[rd]^{\rho}&&Q\times T^\ld\ar@{->}[ld]\\
&Q&}
\]
The induced isomorphism $\xi_*$ identifies $H_*(Y,\dd Y)$ with
$H_*(Q,\dd Q)\otimes \Lambda$ and $H_*(Y)$ with $\Lambda$.
\end{lemma}

\begin{proof}
Principal $T^\ld$-bundles are classified by their Euler classes,
sitting in $H^2(Q; \Zo^\ld)=0$ (recall that $Q$ is acyclic over
$\Zo$). The second statement follows from the K\"{u}nneth
isomorphism.
\end{proof}

For a general principal $T^\ld$-bundle $\rho\colon Y\to Q$
consider the filtration
\begin{equation}\label{eqFiltrationY}
\varnothing=Y_{-1}\subset Y_0\subset Y_1\subset\ldots\subset
Y_{n-1}\subset Y_n=Y,
\end{equation}
where $Y_i=\rho^{-1}(Q_i)$. For each $I\in S_Q$ consider the
subsets $Y_I=\rho^{-1}(F_I)$ and $\dd Y_I=\rho^{-1}(\dd F_I)$. In
particular, $Y_{\varnothing}=Y$, $\dd Y=Y_{n-1}$.

Let $\Ey^*_{*,*}$ be the spectral sequence associated with
filtration \eqref{eqFiltrationY}, i.e.:
\[
\Ey^1_{p,q}\cong H_{p+q}(Y_p,Y_{p-1}) \Rightarrow
H_{p+q}(Y),\qquad \dy^r\colon \Ey^r_{*,*}\to\Ey^r_{*-r,*+r-1}
\]
and $H_{p+q}(Y_p,Y_{p-1})\cong \bigoplus_{|I|=n-p}H_{p+q}(Y_I,\dd
Y_I)$.

Similar to construction \ref{conSheafOnQ} we define the sheaf
$\hh_q^Y$ on $S_Q$ by setting
\begin{equation}\label{eqHhYsheafDef}
\hh_q^Y(I) = H_{q+n-|I|}(Y_I,\dd Y_I).
\end{equation}
The restriction maps coincide with the differential $\dy^1$ up to
incidence signs. Note that $\hh^Y\eqd\bigoplus_q \hh_q^Y$ has a
natural $\Lambda$-module structure induced by the torus action.
The cochain complex of $\hh^Y$ coincides with the first page of
$\Ey^*_{*,*}$ up to change of indices. As before, consider also
the truncated spectral sequence:
\[
\Eyt^1_{p,q}\cong H_{p+q}(Y_p,Y_{p-1}), p<n \Rightarrow
H_{p+q}(\dd Y),
\]
and the truncated sheaf: $\hht^Y_q(\varnothing)=0$,
$\hht^Y_q(I)=\hht^Y(I)$ for $I\neq\varnothing$.

\begin{lemma}\label{lemmaLaDefin}
If $Q$ is Buchsbaum, then $\hht^Y_q\cong \hht_0\otimes \La^{(q)}$,
where $\La^{(q)}$ is a locally constant sheaf on $S_Q$ valued by
$\Lambda^{(q)}$.
\end{lemma}

\begin{proof}
All proper faces of $Q$ are Cohen--Macaulay, thus lemma
\ref{lemmaCMtrivialis} applies. We have $H_q(Y_I,\dd Y_I) \cong
H_0(F_I,\dd F_I)\otimes \Lambda^{(q)}$. For any $I<J$ there are
two trivializations of $Y_J$: the restriction of $\xi_I$, and
$\xi_J$ itself:
\[
\xymatrix{
&Y_J\ar@{^{(}->}[r]\ar@{->}[dl]_{\xi_J}\ar@{->}[d]^{\xi_I|_{Y_J}}&
Y_I\ar@{->}[d]^{\xi_I} \\
F_J\times T^\ld\ar@{..>}[r]& F_J\times
T^\ld\ar@{^{(}->}[r]&F_I\times T^\ld}
\]
Transition maps $\xi_I|_{Y_J}\circ(\xi_J)^{-1}$ induce the
isomorphisms in homology $\Lambda^{(q)}=H_q(F_J\times T^\ld)\to
H_q(F_J\times T^\ld)=\Lambda^{(q)}$ which determine the
restriction maps $\La^{(q)}(I\subset J)$. The locally constant
sheaf $\La^{(q)}$ is thus defined, and the statement follows.
\end{proof}

Denote $\La=\bigoplus_q\La^{(q)}$ --- the graded sheaf on $S_Q$
valued by $\Lambda=\bigoplus_q\Lambda^{(q)}$

\begin{rem}
Our main example is the trivial bundle: $Y=Q\times T^\ld$. In this
case the whole spectral sequence $\Ey^*_{*,*}$ is isomorphic to
$\Eq^*_{*,*}\otimes \Lambda$. For the structure sheaves we also
have $\hh^Y_*=\hh_*\otimes\Lambda^*$. In particular the sheaf
$\La$ constructed in lemma \ref{lemmaLaDefin} is globally trivial.
By results of subsection \ref{subsecSpecSeqQ}, all terms and
differentials of $\Ey^*_{*,*}$ are described explicitly.
Nevertheless, several results of this paper remain valid in a
general setting, thus are stated in full generality where it is
possible.
\end{rem}

\begin{rem}
This construction is very similar to the construction of the sheaf
of local fibers which appears in the Leray--Serre spectral
sequence. But contrary to this general situation, here we
construct not just a sheaf in a common topological sense, but a
cellular sheaf supported on the given simplicial poset $S_Q$. Thus
we prefer to provide all the details, even if they seem obvious to
the specialists.
\end{rem}

\subsection{Torus spaces over simple pseudo-cell complexes}\label{subsecBT}

Recall, that $\Tt_\ld$ denotes the set of all 1-dimensional toric
subgroups of $T^\ld$. Let $Q$ be a simple pseudo-cell complex of
dimension $n$, $S_Q$ --- its underlying simplicial poset and
$\rho\colon Y\to Q$ --- a principal $T^\ld$-bundle. There exists a
general definition of a characteristic pair in the case of
manifolds with locally standard actions, see \cite[Def.4.2]{Yo}.
We do not review this definition here due to its complexity, but
prefer to work in Buchsbaum setting, in which case many things
simplify. If $Q$ is Buchsbaum, then its proper faces $F_I$ are
Cohen--Macaulay, and according to lemma \ref{lemmaCMtrivialis},
there exist trivializations $\xi_I$ which identify orbits over
$x\in F_I$ with $T^\ld$. If $x$ belongs to several faces, then
different trivializations give rise to the transition
homeomorphisms $\tr_{I<J}\colon T^\ld\to T^\ld$, and at the global
level some nontrivial twisting may occur. To give the definition
of characteristic map, we need to distinguish between these
different trivializations. Denote by $T^\ld(I)$ the torus sitting
over the face $F_I$ (via trivialization of lemma
\ref{lemmaCMtrivialis}) and let $\Tt_\ld(I)$ be the set of
1-dimensional subtori of $T^\ld(I)$. The map $\tr_{I<J}$ sends
elements of $\Tt_\ld(I)$ to $\Tt_\ld(J)$ in an obvious way. One
can think of $\Tt_\ld(-)$ as a locally constant sheaf of sets on
$S_Q\setminus\{\varnothing\}$.

\begin{defin}
A characteristic map $\lambda$ is a collection of elements
$\lambda(i)\in\Tt_\ld(i)$ defined for each vertex $i\in
\ver(S_Q)$. This collection should satisfy the following
condition: for any simplex $I\in S_Q$, $I\neq\varnothing$ with
vertices $i_1,\ldots,i_k$ the set
\begin{equation}\label{eqCollectInCharFun}
\{\tr_{i_1<I}\lambda(i_1),\ldots,\tr_{i_k<I}\lambda(i_k)\}
\end{equation}
satisfies $\sta{\ko}$ condition in $T^\ld(I)$.
\end{defin}

Clearly, a characteristic map exists only if $\ld\geqslant n$. Let
$T^{\lambda(I)}$ denote the subtorus of $T^\ld(I)$ generated by
$1$-dimensional subgroups \eqref{eqCollectInCharFun}.

\begin{con}[Quotient construction]
Consider the identification space:
\begin{equation}\label{eqDefX}
X = Y/\simc,
\end{equation}
where $y_1\sim y_2$ if $\rho(y_1)=\rho(y_2)\in F_I\oo$ for some
$\varnothing\neq I\in S_Q$, and $y_1,y_2$ lie in the same
$T^{\lambda(I)}$-orbit.

There is a natural action of $T^\ld$ on $X$ coming from $Y$. The
map $\mu\colon X\to Q$ is a projection to the orbit space
$X/T^\ld\cong Q$. The orbit $\mu^{-1}(b)$ over the point $b\in
F\oo_I\subset \dd Q$ is identified (via the trivializing
homeomorphism) with $T^\ld(I)/T^{\lambda(I)}$. This orbit has
dimension $\ld-\dim T^{\lambda(I)}=\ld-|I|=\dim F_I+(\ld-n)$. The
preimages of points $b\in Q\setminus \dd Q$ are the
full-dimensional orbits.

Filtration \eqref{eqFiltrationY} descends to the filtration on
$X$:
\begin{equation}\label{eqFiltrationX}
\varnothing=X_{-1}\subset X_0\subset X_1\subset\ldots\subset
X_{n-1}\subset X_n=X,
\end{equation}
where $X_i=Y_i/\simc$ for $i\leqslant n$. In other words, $X_i$ is
the union of $(\leqslant i+\ld-n)$-dimensional orbits of the
$T^\ld$-action. Thus $\dim X_i = 2i+\ld-n$ for $i\leqslant n$.
\end{con}

Let $\Ex^*_{*,*}$ be the spectral sequence associated with
filtration \eqref{eqFiltrationX}:
\[
\Ex^1_{p,q}=H_{p+q}(X_p,X_{p-1})\Rightarrow H_{p+q}(X),\qquad
\dx\colon \Ex^r_{*,*}\to\Ex^r_{*-r,*+r-1}.
\]
The quotient map $f\colon Y \to X$ induces a morphism of spectral
sequences $f_*^r\colon \Ey^r_{*,*}\to \Ex^r_{*,*}$, which is a
$\Lambda$-module homomorphism for each $r\geqslant 1$.

\subsection{Structure of $\Ex^1_{*,*}$.}\label{subsecStructOfEX1}

For each $I\in S_Q$ consider the subsets $X_I=Y_I/\simc$ and $\dd
X_I=\dd Y_I/\simc$. As before, define the family of sheaves
associated with filtration \ref{eqFiltrationX}:
\[
\hh^X_q(I) = H_{q+n-|I|}(X_I,\dd X_I),
\]
with the restriction maps equal to $\dy^1$ up to incidence signs.
These sheaves can be considered as a single sheaf $\hh^X$ graded
by $q$. We have
$(\Ex^1_{*,q},\dx)\cong(C^{n-1-*}(S_Q,\hh_q^X),d)$. There are
natural morphisms of sheaves $f_*\colon\hh^Y_q\to\hh^X_q$ induced
by the quotient map $f\colon Y\to X$, and the corresponding map of
cochain complexes coincides with $f_*^1\colon
\Ey^1_{*,q}\to\Ex^1_{*,q}$. Also consider the truncated versions:
$\hht^X=\bigoplus_q\hht^X_q$ for which $\hht^X(\varnothing)=0$.

\begin{rem}\label{remXYsameForPeqN}
The map $f_*^1\colon H_*(Y,\dd Y) \to H_*(X,\dd X)$ is an
isomorphism by excision since $X/\dd X\cong Y/\dd Y$.
\end{rem}

Now we describe the truncated part of the sheaf $\hh^Y$ in
algebraical terms. Let $I\in S_Q$ be a simplex and $i\leqslant I$
its vertex. Consider the element of exterior algebra
$\omega_{i}\in \La(I)^{(1)}\cong \Lambda^{(1)}$ which is the image
of the fundamental cycle of $\lambda(i)\cong T^1$ under the
transition map $\tr_{i\leqslant I}$:
\begin{equation}\label{eqDefOmegaForm}
\omega_i=\left(\tr_{i\leqslant I}\right)_*[\lambda(i)]\in
\La(I)^{(1)}
\end{equation}
Consider the subsheaf $\I$ of $\La$ whose value on a simplex $I$
with vertices $\{i_1,\ldots,i_k\}\neq\varnothing$ is:
\begin{equation}\label{eqDefIdealSheaf}
\I(I)=(\omega_{i_1},\ldots,\omega_{i_k})\subset \La(I),
\end{equation}
--- the ideal of the exterior algebra $\La(I)\cong \Lambda$
generated by linear forms. Also set $\I(\varnothing)=0$. It is
easily checked that $\La(I<J)\I(I)\subset\I(J)$, so $\I$ is a
well-defined subsheaf of $\La$.

\begin{lemma}\label{lemmaTruncExtDescr}
The map of sheaves $f_*\colon\hht^Y_q\to\hht^X_q$ is isomorphic to
the quotient map of sheaves $\hh_0\otimes
\La^{(q)}\to\hh_0\otimes(\La/\I)^{(q)}$.
\end{lemma}

\begin{proof}
By lemma \ref{lemmaCMtrivialis}, $(Y_I,\dd Y_I)\to (F_I,\dd F_I)$
is equivalent to the trivial $T^\ld$-bundle $\xi_I\colon(Y_I,\dd
Y_I)\cong(F_I,\dd F_I)\times T^\ld(I)$. By construction of $X$, we
have identifications
\[
\xi_I'\colon(X_I,\dd X_I)\cong \left[(F_I,\dd F_I)\times
T^\ld(I)\right]/\simc.
\]
By excision, the group $H_*([F_I\times T^\ld(I)]/\simc,[\dd
F_I\times T^\ld(I)]/\simc)$ coincides with
\[
H_*(F_I\times T^\ld(I)/T^{\lambda(I)},\dd F_I\times
T^\ld(I)/T^{\lambda(I)})= H_*(F_I,\dd F_I)\otimes
H_*(T^\ld(I)/T^{\lambda(I)}).
\]
The rest follows from lemma \ref{lemmaIdealQuotTorus}.
\end{proof}

There is a short exact sequence of graded sheaves
\[
0\longrightarrow \I \longrightarrow \La \longrightarrow \La/\I
\longrightarrow 0
\]
Tensoring it with $\hh_0$ produces the short exact sequence
\[
0\longrightarrow \hh_0\otimes\I^{(q)} \longrightarrow \hht^Y_q
\longrightarrow \hht^X_q \longrightarrow 0
\]
according to lemma \ref{lemmaTruncExtDescr}. The sheaf
$\hh_0\otimes\I$ can also be considered as a subsheaf of
non-truncated sheaf $\hh^Y$.

\begin{lemma}\label{lemmaIshortNonTrunc}
There is a short exact sequence of graded sheaves
\[
0\rightarrow \hh_0\otimes\I\rightarrow \hh^Y \rightarrow \hh^X
\rightarrow 0.
\]
\end{lemma}

\begin{proof}
Follows from the diagram
\[
\xymatrix{
&&0\ar@{->}[d]&0\ar@{->}[d]&\\
0\ar@{->}[r]&\hh_0\otimes\I\ar@{^{(}->}[r]\ar@{=}[d]&
\hht^Y\ar@{->>}[r]\ar@{^{(}->}[d]&\hht^X\ar@{->}[r]\ar@{^{(}->}[d]&0\\
0\ar@{->}[r]&\hh_0\otimes\I\ar@{^{(}->}[r]&\hh^Y\ar@{->>}[r]\ar@{->>}[d]&
\hh^X\ar@{->}[r]\ar@{->>}[d]&0\\
&&\hh^Y/\hht^Y\ar@{->}[r]^{\cong}\ar@{->}[d]&\hh^X/\hht^X\ar@{->}[d]&\\
&&0&0&
}
\]
The lower sheaves are concentrated in $\varnothing\in S_Q$ and the
graded isomorphism between them is due to remark
\ref{remXYsameForPeqN}.
\end{proof}

\subsection{Extra pages of $\Ey$ and
$\Ex$}\label{subsecExtraTermX}

To simplify further discussion we briefly sketch the formalism of
additional pages of spectral sequences $\Ey$ and $\Ex$, which
extends considerations of subsection \ref{subsecExtraTermQ}.
Consider the following bigraded module:
\[
\Ey^{1+}_{p,q}=\begin{cases} \Eyt_{p,q}^2, \mbox{ if }p<n;\\
\Ey^1_{n,q}, \mbox{ if }p=n.
\end{cases}
\]
and define the differentials $\dy^{1-}$ on $\Ey^1$ and $\dy^{1+}$
on $\Ey^{1+}$ by
\[
\dy^{1-}=\begin{cases} \dy^1, \mbox{ if }p<n;\\
0,\mbox{ if }p=n.
\end{cases}\qquad
\dy^{1+}=\begin{cases} 0, \mbox{ if }p<n;\\
\Ey^1_{n,q}\stackrel{\dy^1}{\longrightarrow}\Ey^1_{n-1,q}
\longrightarrow \Ey^2_{n-1,q}\mbox{ if }p=n.
\end{cases}
\]
It is easily checked that $\Ey^{1+}\cong H(\Ey^1,\dy^{1-})$ and
$\Ey^2\cong H(\Ey^{1+},\dy^{1+})$. The page $\Ex^{1+}$ and the
differentials $\dx^{1-}$, $\dx^{1+}$ are defined similarly. The
map $f_*^1\colon \Ey^1\to\Ex^1$ induces the map between the extra
pages: $f_*^{1+}\colon \Ey^{1+}\to \Ex^{1+}$.

\section{Main results}\label{secMainResults}

\subsection{Structure of $\Ex^r_{*,*}$}
The short exact sequence of lemma \ref{lemmaIshortNonTrunc}
generates the long exact sequence in sheaf cohomology:
\begin{equation}\label{eqLongExact}
\rightarrow H^{i-1}(S_Q;\hh_0\otimes\I^{(q)}) \rightarrow
H^{i-1}(S_Q;\hh^Y_q) \stackrel{f^2_*}{\longrightarrow}
H^{i-1}(S_Q;\hh^X_q) \longrightarrow
H^{i}(S_Q;\hh_0\otimes\I^{(q)})\rightarrow
\end{equation}

The following lemma is the cornerstone of the whole work.

\begin{lemma}[Key Lemma]\label{lemmaKeyLemma}
$H^i(S_Q;\hh_0\otimes\I^{(q)})=0$ if $i\leqslant n-1-q$.
\end{lemma}

The proof follows from a more general sheaf-theoretical fact and
is postponed to section \ref{SecKeyLemma}. In the following we
simply write $S$ instead of $S_Q$. By construction,
$\Ey^2_{p,q}\cong H^{n-1-p}(S;\hh^Y_q)$ and $\Ex^2_{p,q}\cong
H^{n-1-p}(S;\hh^X_q)$. The Key lemma \ref{lemmaKeyLemma} and exact
sequence \eqref{eqLongExact} imply

\begin{lemma}\label{lemIsomInject}
\begin{gather*}
f_*^2\colon \Ey^2_{p,q}\to\Ex^2_{p,q} \mbox{ is an isomorphism
if }p>q,\\
f_*^2\colon \Ey^2_{p,q}\to\Ex^2_{p,q} \mbox{ is injective if }p=q.
\end{gather*}
\end{lemma}

In case $\ld=n$ this observation immediately describes
$\Ex^2_{*,*}$ in terms of $\Ey^2_{*,*}$. Under the notation
\[
(\dy)^r_{p,q}\colon\Ey^r_{p,q}\to\Ey^r_{p-r,q+r-1},\qquad
(\dx)^r_{p,q}\colon\Ex^r_{p,q}\to\Ex^r_{p-r,q+r-1}
\]
there holds

\begin{thm}\label{thmTwoSpecSeqGeneral}
Let $Q$ be Buchsbaum pseudo-cell complex of dimension $n$, $Y$ be
a principal $T^n$-bundle over $Q$, $f\colon Y\to X= Y/\simc$ ---
the quotient construction, and
$f_*^r\colon\Ey^r_{*,*}\to\Ex^r_{*,*}$
--- the induced map of homological spectral sequences associated
with filtrations \ref{eqFiltrationY}, \ref{eqFiltrationX}. Then
\[
f_*^2\colon\Ey^2_{p,q}\to \Ex^2_{p,q}\mbox{ is }\begin{cases}
\mbox{an isomorphism if } q<p \mbox{ or } q=p=n,\\
\mbox{injective if } q=p<n,
\end{cases}
\]
and $\Ex^2_{p,q}=0$ if $q>p$. Higher differentials of
$\Ex^*_{*,*}$ thus have the form
\[
(\dx)^r_{p,q}=\begin{cases}
f_*^r\circ(\dy)^r_{p,q}\circ(f_*^r)^{-1},
\mbox{ if } p-r\geqslant q+r-1,\\
0\mbox{ otherwise},
\end{cases}
\]
for $r\geqslant 2$.
\end{thm}

If $Y$ is a trivial $T^n$-bundle, then the structure of
$\Ey^*_{*,*}\cong \Eq^*_{*,*}\otimes\Lambda$ is described
completely by subsection \ref{subsecSpecSeqQ}. In this case almost
all the terms of $\Ex^*_{*,*}$ are described explicitly.

\begin{thm}\label{thmEX2structure}
In the notation of Theorem \ref{thmTwoSpecSeqGeneral} suppose
$Y=Q\times T^n$. Let $\Lambda^{(q)}=H_q(T^n)$ and $\delta_i\colon
H_i(Q,\dd Q) \to H_{i-1}(\dd Q)$ be the connecting homomorphisms.
Then
\begin{equation}\label{eqEX2structure}
\Ex^2_{p,q}\cong\begin{cases} H_p(\dd Q)\otimes\Lambda^{(q)},
\mbox{ if }
q<p\leqslant n-2;\\
\coker\delta_n\otimes\Lambda^{(q)}, \mbox{ if } q<p=n-1;\\
\Ker\delta_n\otimes \Lambda^{(q)}\oplus
\left(\bigoplus\limits_{\substack{q_1+q_2=n+q\\q_1<n}}H_{q_1}(Q,\dd
Q) \otimes \Lambda^{(q_2)}\right), \mbox{ if } q<p=n;\\
H_n(Q,\dd Q) \otimes \Lambda^{(n)},\mbox{ if } q=p=n;\\
0, \mbox{ if } q>p.
\end{cases}
\end{equation}
The maps $f_*^2\colon H_q(\dd
Q)\otimes\Lambda^{(q)}\hookrightarrow \Ex^2_{q,q}$ are injective
for $q<n-1$. Higher differentials for $r\geqslant 2$ are the
following:
\[
\dx^r\cong\begin{cases} \delta_{q_1}\otimes\id_{\Lambda}\colon
\overset{\substack{\Ex^*_{n,q_1+q_2-n}\\
\cup}}{H_{q_1}(Q,\dd Q)\otimes\Lambda^{(q_2)}}\to \overset{\substack{\Ex^*_{q_1-1,q_2}\\
\cup}}{H_{q_1-1}(\dd Q)\otimes\Lambda^{(q_2)}}, \\
\hfill \mbox{
if } r=n-q_1+1, q_1-1>q_2;\\
f_*^2\circ(\delta_{q_1}\otimes \id_{\Lambda})\colon H_{q_1}(Q,\dd
Q)\otimes\Lambda^{(q_2)}\to H_{q_1-1}(\dd
Q)\otimes\Lambda^{(q_2)}\hookrightarrow \Ex^2_{q_1-1,q_1-1},\qquad
\\ \hfill \mbox{ if } r=n-q_1+1, q_1-1=q_2;\\
0, \mbox{ otherwise}.
\end{cases}
\]
\end{thm}

Using the formalism of extra pages introduced in subsection
\ref{subsecExtraTermX}, Theorem \ref{thmEX2structure} can be
restated in a more convenient and concise form

\begin{stm}\label{stmThmSpecConvint}
There exists a spectral sequence
whose first term is
\begin{equation}\label{eqEXstructureFormal}
\Ex^{1+}_{p,q}\cong\begin{cases} H_p(\dd
Q)\otimes\Lambda^{(q)},\mbox{
if }q<p<n;\\
\bigoplus\limits_{q_1+q_2=q+n}H_{q_1}(Q,\dd
Q)\otimes\Lambda^{(q_2)},\mbox{ if }p=n;\\
0, \mbox{ if }q>p;
\end{cases}
\end{equation}
and subsequent terms coincide with $\Ex^r_{*,*}$ for $r\geqslant
2$. There exist injective maps $f_*^{1+}\colon H_q(\dd
Q)\otimes\Lambda^{(q)}\hookrightarrow \Ex^2_{q,q}$ for $q<n$.
Differentials for $r\geqslant 1$ have the form
\[
\dx^r\cong\begin{cases} \delta_{q_1}\otimes\id_{\Lambda}\colon
\overset{\substack{\Ex^*_{n,q_1+q_2-n}\\
\cup}}{H_{q_1}(Q,\dd Q)\otimes\Lambda^{(q_2)}}\to \overset{\substack{\Ex^*_{q_1-1,q_2}\\
\cup}}{H_{q_1-1}(\dd Q)\otimes\Lambda^{(q_2)}}, \\
\hfill \mbox{
if } r=n-q_1+1, q_1-1>q_2;\\
f_*^{1+}\circ(\delta_{q_1}\otimes \id_{\Lambda})\colon
H_{q_1}(Q,\dd Q)\otimes\Lambda^{(q_2)}\to H_{q_1-1}(\dd
Q)\otimes\Lambda^{(q_2)}\hookrightarrow \Ex^2_{q_1-1,q_1-1},\qquad
\\ \hfill \mbox{ if } r=n-q_1+1, q_1-1=q_2;\\
0, \mbox{ otherwise}.
\end{cases}
\]
\end{stm}

Note that the terms $\Ex^*_{q,q}$ for $q<n$ are not mentioned in
the lists \eqref{eqEX2structure}, \eqref{eqEXstructureFormal}. Let
us call $\bigoplus_{q<n}\Ex^{1+}_{q,q}$ the \emph{border} of
$\Ex^*_{*,*}$. This name is due to the fact that all entries above
the border vanish: $\Ex^*_{p,q}=0$ for $q>p$.

Denote $\dim_{\ko}\Hr_p(S)=\dim_{\ko}\Hr_p(\dd Q)$ by $\br_p(S)$
for $p<n$. The ranks of the border components are described as
follows:

\begin{thm}\label{thmBorderStruct}
In the notation of Theorem \ref{thmEX2structure} and statement
\ref{stmThmSpecConvint}
\[
\dim \Ex^{1+}_{q,q} = h_q(S)+{n\choose
q}\sum\limits_{p=0}^{q}(-1)^{p+q}\br_p(S)
\]
for $q\leqslant n-1$, where $h_q(S)$ are the $h$-numbers of the
simplicial poset $S$.
\end{thm}

\begin{thm}\label{thmBorderManif}\mbox{}

(1) Let $Q$ be a Buchsbaum manifold over $\ko$. Then $\dim
\Ex^{1+}_{q,q}=h'_{n-q}(S)$ for $q\leqslant n-2$ and
$\Ex^{1+}_{n-1,n-1}=h'_1(S)+n$.

(2) Let $Q$ be Buchsbaum manifold such that $H_n(Q,\dd Q)\cong\ko$
and $\delta_n\colon H_n(Q, \dd Q)\to H_{n-1}(\dd Q)$ is injective.
Then $\dim \Ex^2_{q,q} = h'_{n-q}(S)$ for $0\leqslant q\leqslant
n$.
\end{thm}

The definitions of $h$-, $h'$- and $h''$-vectors and the proof of
Theorems \ref{thmBorderStruct},\ref{thmBorderManif} and
\ref{thmHtwoPrimes} are gathered in section \ref{secHvectors}.
Note that it is sufficient to prove Theorem \ref{thmBorderStruct}
and the first part of Theorem \ref{thmBorderManif} in the case
$Q=P(S)$. Indeed, by definition, $\Ex^{1+}_{q,q}=\Ext^2_{q,q}$ for
$q\leqslant n-1$, and there exists a map $(Q\times T^n)/\simc\to
(P(S_Q)\times T^n)/\simc$ which covers the map $\varphi$ of lemma
\ref{lemmaUniverPosets} and induces the isomorphism of
corresponding truncated spectral sequences.

In the cone case the border components can be described explicitly
up to $\infty$-term.

\begin{thm}\label{thmHtwoPrimes}
Let $S$ be a Buchsbaum poset, $Q=P(S)$, $Y=Q\times T^n$,
$X=Y/\simc$, $\Ex^r_{p,q}\Rightarrow H_{p+q}(X)$ --- the
homological spectral sequence associated with filtration
\eqref{eqFiltrationX}. Then
\[
\dim\Ex^{\infty}_{q,q}=h''_q(S)
\]
for $0\leqslant q\leqslant n$.
\end{thm}

\begin{cor}
If $S$ is Buchsbaum, then $h''_i(S)\geqslant 0$.
\end{cor}

\begin{proof}
For any $S$ there exists a characteristic map over $\Qo$. Thus
there exists a space $X=(P(S)\times T^n)/\simc$ and Theorem
\ref{thmHtwoPrimes} applies.
\end{proof}

\subsection{Homology of $X$}
Theorem \ref{thmEX2structure} implies the additional grading on
$H_*(S)$ --- the one given by the degrees of exterior forms. It is
convenient to work with this double grading.

\begin{con}
Suppose $Y=Q\times T^n$. For $j\in[0,n]$ consider the \ang-shaped
spectral sequence
\[
\Eyj^r_{*,*}=\Eq^r_{*,*}\otimes \Lambda^{(j)}.
\]
Clearly, $\Ey^r_{*,*}=\bigoplus_{j=0}^n\Eyj^r_{*,*}$ and
$\Eyj^*_{p,q}\Rightarrow H_{p+q-j,j}(Y)\eqd
H_{p+q-j}(Q)\otimes\Lambda^{(j)}$. In particular,
\[
\Eyj^{1+}_{*,*}=\bigoplus_{p<n}\left(\Eq^{1+}_{p,0}
\otimes\Lambda^{(j)}\right)\oplus\bigoplus_q
\left(\Eq^{1+}_{n,q}\otimes \Lambda^{(j)}\right)
\]
Consider the corresponding \ang-shaped spectral subsequences in
$\Ex^*_{*,*}$. Start with the $\ko$-modules:
\[
\Exj^{1+}_{*,*}=\bigoplus_{p<n}\Ex^{1+}_{p,j}\oplus\bigoplus_q
f_*^{1+}\left(\Eq^{1+}_{n,q}\otimes \Lambda^{(q)}\right).
\]
By statement \ref{stmThmSpecConvint}, all the differentials of
$\Ex^*_{*,*}$ preserve $\Exj^*_{*,*}$, thus the spectral
subsequences $\Exj^r_{*,*}$ are well defined, and
$\Ex^r_{*,*}=\bigoplus_{j=0}^n\Exj^r_{*,*}$. Let $H_{i,j}(X)$ be
the family of subgroups of $H_*(X)$ such that
$\Exj^r_{p,q}\Rightarrow H_{p+q-j,j}(X)$. Then
\[
H_k(X)=\bigoplus_{i+j=k}H_{i,j}(X)
\]
and the map $f_*\colon H_*(Y)\to H_*(X)$ sends $H_{i,j}(Y)\cong
H_i(Q)\otimes \Lambda^{(j)}$ to $H_{i,j}(X)$. The map $f_*^r\colon
\Ey^r\to\Ex^r$ sends $\Eyj^r$ to $\Exj^r$ for each
$j\in\{0,\ldots,n\}$ and we have commutative squares:
\[
\xymatrix{ \Eyj^r_{p,q}\ar@{=>}[r]\ar@{->}[d]^{f^r_*}&
H_{p+q-j,j}(Y)\ar@{->}[d]^{f_*}\\
\Exj^r_{p,q}\ar@{=>}[r]&H_{p+q-j,j}(X)}
\]
\end{con}

\begin{prop}\label{propHomologyX}\mbox{}

\begin{enumerate}
\item If $i>j$, then $f_*\colon H_{i,j}(Y)\to
H_{i,j}(X)$ is an isomorphism. In particular, $H_{i,j}(X)\cong
H_i(Q)\otimes\Lambda^{(j)}$.

\item If $i<j$, then there exists an isomorphism $H_{i,j}(X)\cong H_i(Q,\dd
Q)\otimes \Lambda^{(j)}$.

\item In case $i=j<n$, the module $H_{i,i}(X)$ fits in the exact sequence
\[
0\rightarrow\Ex^{\infty}_{i,i}\rightarrow H_{i,i}(X)\rightarrow
H_i(Q,\dd Q)\otimes\Lambda^{(i)} \rightarrow 0,
\]
or, equivalently,
\[
0\rightarrow \im\delta_{i+1}\otimes \Lambda^{(i)}\rightarrow
\Ex^{1+}_{i,i}\rightarrow H_{i,i}(X)\rightarrow H_i(Q,\dd
Q)\otimes\Lambda^{(i)}\rightarrow 0
\]

\item If $i=j=n$, then $H_{n,n}(X)=\Ex^{\infty}_{n,n}=\Ex^1_{n,n}$.
\end{enumerate}
\end{prop}

\begin{proof}
According to statement \ref{stmThmSpecConvint}
\begin{equation}\label{eqLongGradeIso}
f_*^{1+}\colon\Eyj^{1+}_{i,q}\to\Exj^{1+}_{i,q} \mbox{ is
}\begin{cases} \mbox{the isomorphism if }i>j \mbox{ or }i=j=n;\\
\mbox{injective if }i=j.
\end{cases}
\end{equation}
For each $j$ both spectral sequences $\Eyj$ and $\Exj$ are
\ang-shaped, thus fold in the long exact sequences:
\begin{equation}\label{eqTwoLongGrad}
\xymatrix{\ldots\ar@{->}[r]&\Eyj^{1+}_{i,j}\ar@{->}[r]\ar@{->}[d]^{f_*}&
H_{i,j}(Y)\ar@{->}[r]\ar@{->}[d]^{f_*}&\Eyj^{1+}_{n,i-n+j}
\ar@{->}[r]^{\dy^{n-i+1}}\ar@{->}[d]^{f_*}_{\cong}&
\Eyj^{1+}_{i-1,j}\ar@{->}[r]\ar@{->}[d]^{f_*}&\ldots\\
\ldots\ar@{->}[r]&\Exj^{1+}_{i,j}\ar@{->}[r]&
H_{i,j}(X)\ar@{->}[r]&\Exj^{1+}_{n,i-n+j}\ar@{->}[r]^{\dx^{n-i+1}}&
\Exj^{1+}_{i-1,j}\ar@{->}[r]&\ldots}
\end{equation}
Applying five lemma in the case $i>j$ proves (1). For $i<j$, the
groups $\Exj^{1+}_{i,j}$, $\Exj^{1+}_{i-1,j}$ vanish by
dimensional reasons thus $H_{i,j}(X) \cong \Exj^{1+}_{n,i-n+j}
\cong \Eyj^{1+}_{n,i-n+j} \cong H_i(Q,\dd Q)\otimes\Lambda^{(j)}$.
Case $i=j$ also follows from \eqref{eqTwoLongGrad} by a simple
diagram chase.
\end{proof}

In the manifold case proposition \ref{propHomologyX} reveals a
bigraded duality. If $Q$ is a nice manifold with corners,
$Y=Q\times T^n$ and $\lambda$ is a characteristic map over $\Zo$,
then $X$ is a manifold with locally standard torus action. In this
case Poincare duality respects the double grading.

\begin{prop}\label{propXmanifoldDuality}
If $X=(Q\times T^n)/\simc$ is a manifold with locally standard
torus action and $\ko$ is a field, then $H_{i,j}(X)\cong
H_{n-i,n-j}(X)$.
\end{prop}

\begin{proof}
If $i<j$, then $H_{i,j}(X)\cong H_i(Q,\dd
Q)\otimes\Lambda^{(j)}\cong H_{n-i}(Q)\otimes\Lambda^{(n-j)}\cong
H_{n-i,n-j}(X)$, since $H_i(Q,\dd Q)\cong H_{n-i}(Q)$ by the
Poincare--Lefschetz duality and $H_j(T^n)\cong H_{n-j}(T^n)$ by
the Poincare duality for the torus. The remaining isomorphism
$H_{i,i}(X)\cong H_{n-i,n-i}(X)$ now follows from the ordinary
Poincare duality for $X$.
\end{proof}

\begin{rem}\label{remOrbifoldDuality}
If the space $X=(Q\times T^n)/\simc$ is constructed from a
manifold with corners $(Q,\dd Q)$ using characteristic map over
$\Qo$ (i.e. $X$ is a toric orbifold), then proposition
\ref{propXmanifoldDuality} still holds over $\Qo$.
\end{rem}

%

\section{Duality between certain cellular sheaves and cosheaves}\label{SecKeyLemma}

\subsection{Proof of the Key lemma}
In this section we prove lemma \ref{lemmaKeyLemma}. First recall
the setting.
\begin{itemize}
\item $Q$ : a Buchsbaum pseudo-cell complex with the underlying simplicial
poset $S=S_Q$ (this poset is Buchsbaum itself by corollary
\ref{corQbuchSbuch}).

\item $\hh_0$ : the structure sheaf on
$S$; $\hh_0(J)=H_{\dim F_J}(F_J,\dd F_J)$.

\item $\La$ : a locally constant graded sheaf on $S$ valued by exterior
algebra $\Lambda=H_*(T^\ld)$. This sheaf is associated in a
natural way to a principal $T^\ld$-bundle over $Q$,
$\La(J)=H_*(T^\ld(J))$ for $J\neq \varnothing$ and
$\La(\varnothing)=0$. By inverting all restriction maps we obtain
the cosheaf $\Lah$.

\item $\lambda$ : a characteristic map over $S$. It determines
$T^1$-subgroup $\tr_{i\leqslant J}(\lambda(i))\subset T^\ld(J)$
for each simplex $J$ with vertex $i$. The homology class of this
subgroup is denoted $\omega_i\in \La(I)$ (see
\eqref{eqDefOmegaForm}). Note that the restriction isomorphism
$\La(J_1<J_2)$ sends $\omega_i\in \La(J_1)$ to
$\omega_i\in\La(J_2)$. Thus we simply write $\omega_i$ for all
such elements since the ambient exterior algebra will be clear
from the context.

\item $\I$ : the sheaf of ideals, associated to $\lambda$. The
value of $\I$ on a simplex $J\neq\varnothing$ with vertices
$\{i_1,\ldots,i_k\}$ is the ideal
$\I(J)=(\omega_{i_1},\ldots,\omega_{i_k})\subset \La(J)$. Clearly,
$\I$ is a graded subsheaf of $\La$.
\end{itemize}

We now introduce another type of ideals.

\begin{con}
Let $J=\{i_1,\ldots,i_k\}$ be a nonempty subset of vertices of
simplex $I\in S$. Consider the element $\pi_J\in\La(I)$,
$\pi_J=\bigwedge_{i\in J}\omega_i$. By the definition of
characteristic map, the elements $\omega_i$ are linearly
independent, thus $\pi_J$ is a non-zero $|J|$-form. Let
$\Pi_J\subset \La(I)$ be the principal ideal generated by $\pi_J$.
The restriction maps $\La(I<I')$ identify $\Pi_J\subset \La(I)$
with $\Pi_J\subset \La(I')$.

In particular, when $J$ is the whole set of vertices of a simplex
$I\neq\varnothing$ we define $\Pih(I)\eqd\Pi_I\subset\Lah(I)$. If
$I'<I$, then the corestriction map $\Lah(I>I')=\La(I'<I)^{-1}$
injects $\Pih(I)$ into $\Pih(I')$, since $\Lah(I>I')\pi_I$ is
divisible by $\pi_{I'}$. Thus $\Pih$ is a well-defined graded
subcosheaf of $\Lah$. Formally set $\Pih(\varnothing)=0$.
\end{con}

\begin{thm}\label{thmDuality}
For Buchsbaum pseudo-cell complex $Q$ and $S=S_Q$ there exists an
isomorphism $H^k(S;\hh_0\otimes\I)\cong H_{n-1-k}(S;\Pih)$ which
respects the gradings of $\I$ and $\Pih$.
\end{thm}

Before giving a proof let us deduce the Key lemma. We need to show
that $H^i(S;\hh_0\otimes\I^{(q)})=0$ for $i\leqslant n-1-q$.
According to Theorem \ref{thmDuality} this is equivalent to

\begin{lemma}
$H_i(S;\Pih^{(q)})=0$ for $i\geqslant q$.
\end{lemma}

\begin{proof}
The ideal $\Pih(I)=\Pi_I$ is generated by the element $\pi_I$ of
degree $|I|=\dim I+1$. Thus $\Pi_I^{(q)}=0$ for $q\leqslant \dim
I$. Hence the corresponding part of the chain complex vanishes.
\end{proof}

\begin{proof}[Proof of Theorem \ref{thmDuality}]
The idea of proof is the following. First we construct a
resolution of sheaf $\hh_0\otimes\I$ whose terms are ``almost
acyclic''. By passing to cochain complexes this resolution
generates a bicomplex $\Cc^{\bullet}_{\bullet}$. By considering
two standard spectral sequences for this bicomplex we prove that
both $H^k(S;\hh_0\otimes\I)$ and $H_{n-1-k}(S;\Pih)$ are
isomorphic to the cohomology of the totalization
$\Cc^{\bullet}_{\Tot}$.

For each $\varnothing\neq I\in S$ consider the sheaf $\R_I =
\hh_0\otimes\ups{I}^{\Pi_I}$ (see examples \ref{exSheafLocal} and
\ref{exSheafLocHomol}), i.e.:
\[
\R_I(J) = \begin{cases} \hh_0(J)\otimes\Pi_I,\mbox{ if } I\leqslant J;\\
0\mbox{ otherwise.}
\end{cases}
\]
The sheaf $\R_I$ is graded by degrees of exterior forms:
$\R_I=\bigoplus_q\R_I^{(q)}$. Since $I>I'$ implies
$\Pi_I\subset\Pi_{I'}$, and $i\in\ver(S)$, $i\leqslant J$ implies
$\Pi_{i}\subset \I(J)$, there exist natural injective maps of
sheaves:
\[
\theta_{I>I'}\colon \R_I\hookrightarrow \R_{I'},
\]
and
\[
\eta_{i}\colon \R_i\hookrightarrow \hh_0\otimes\I.
\]
For each $k\geqslant 0$ consider the sheaf
\[
\R_{-k}=\bigoplus_{\dim I = k}\R_I,
\]
These sheaves can be arranged in the sequence
\begin{equation}\label{eqLongSeqOfSheaves}
\R_{\bullet}\colon\qquad\ldots\longrightarrow\R_{-2}\stackrel{d_H}{\longrightarrow}
\R_{-1}\stackrel{d_H}{\longrightarrow}\R_{0}\stackrel{\eta}{\longrightarrow}
\hh_0\otimes\I\longrightarrow 0,
\end{equation}
where $d_H=\bigoplus_{I>_1I'}[I:I']\theta_{I>I'}$ and
$\eta=\bigoplus_{i\in\ver(S)}\eta_i$. By the standard argument
involving incidence numbers $[I:I']$ one shows that
\eqref{eqLongSeqOfSheaves} is a differential complex of sheaves.
Moreover,

\begin{lemma}\label{lemmaExactSeqOfSheaves}
The sequence $\R_{\bullet}$ is exact.
\end{lemma}

\begin{proof}
We should prove that the value of $\R_{\bullet}$ at each $J\in S$
is exact. Since $\R_I(J)\neq 0$ only if $I\leqslant J$ the complex
$\R_{\bullet}(J)$ has the form
\begin{equation}\label{eqTaylorLike}
\ldots\longrightarrow \bigoplus_{I\leqslant J, |I|=2}\Pi_I
\longrightarrow \bigoplus_{I\leqslant J, |I|=1}\Pi_I
\longrightarrow \I(J)\longrightarrow 0,
\end{equation}
tensored with $\hh_0(J)$. Without lost of generality we forget
about $\hh_0(J)$. Maps in \eqref{eqTaylorLike} are given by
inclusions of sub-ideals (rectified by incidence signs). This
looks very similar to the Taylor resolution of monomial ideal in
commutative polynomial ring, but our situation is a bit different,
since $\Pi_I$ are not free modules over $\Lambda$. Anyway, the
proof is similar to commutative case: exactness of
\eqref{eqTaylorLike} follows from inclusion-exclusion principle.
To make things precise (and also to tackle the case $\ko=\Zo$) we
proceed as follows.

By $\sta{\ko}$-condition, the subspace $\langle\omega_{j}\mid j\in
J\rangle$ is a direct summand in $\La^{(1)}(J)\cong \ko^n$. Let
$\{\nu_1,\ldots,\nu_n\}$ be such a basis of $\La^{(1)}(J)$, that
its first $|J|$ vectors are identified with $\omega_j$, $j\in J$.
We simply write $J$ for $\{1,\ldots,|J|\}\subseteq [n]$ by abuse
of notation. The module $\Lambda$ splits in the multidegree
components: $\Lambda=\bigoplus_{A\subseteq[n]}\Lambda_{A}$, where
$\Lambda_A$ is a 1-dimensional $\ko$-module generated by
$\bigwedge_{i\in A}\nu_i$. All modules and maps in
\eqref{eqTaylorLike} respect this splitting. Thus
\eqref{eqTaylorLike} can be written as
\begin{gather*}
\ldots\longrightarrow \bigoplus_{I\subseteq J, |I|=2}
\bigoplus_{A\supseteq I}\Lambda_A \longrightarrow
\bigoplus_{I\subseteq J, |I|=1}\bigoplus_{A\supseteq I}\Lambda_A
\longrightarrow
\bigoplus_{A\cap J\neq\varnothing}\Lambda_A\longrightarrow 0,\\
\bigoplus_{A, A\cap J\neq \varnothing}\left(\ldots\longrightarrow
\bigoplus_{I\subseteq A\cap J, |I|=2} \Lambda_A \longrightarrow
\bigoplus_{I\subseteq A\cap J, |I|=1} \Lambda_A \longrightarrow
\Lambda_A\longrightarrow 0 \right)
\end{gather*}
For each $A$ the cohomology of the complex in brackets coincides
with $\Hr_*(\Delta_{A\cap J};\Lambda_A)\cong\Hr_*(\Delta_{A\cap
J};\ko)$, the reduced simplicial homology of a simplex on the set
$A\cap J\neq\varnothing$. Thus homology vanishes.
\end{proof}

By passing to cochains and forgetting the last term we get a
complex of complexes
\begin{equation}\label{eqCbicomplex}
\Cc^{\bullet}_{\bullet}=\Cc(S,\R_{\bullet})\colon\qquad\ldots\longrightarrow
\Cc^{\bullet}(S;\R_{-2}) \stackrel{d_H}{\longrightarrow}
\Cc^{\bullet}(S;\R_{-1}) \stackrel{d_H}{\longrightarrow}
\Cc^{\bullet}(S;\R_{0}) \longrightarrow 0,
\end{equation}
whose horizontal cohomology vanishes except for the upmost right
position. Let $d_V$ be the ``vertical'' cohomology differential
operating in each $\Cc^{\bullet}(S;\R_{k}^{(q)})$. Then
$d_Hd_V=d_Vd_H$. Thus $\Cc^{\bullet}_{\bullet}$ can be considered
as a bicomplex $(\Cc^{\bullet}_{\Tot},D)$:
\[
\Cc^{\bullet}_{\Tot}=\bigoplus_{i}\Cc^{i}_{\Tot},\quad
\Cc^{i}_{\Tot}=\bigoplus_{k+l=i}\Cc^{l}_{k},\quad D=d_H+(-1)^kd_V.
\]

There are two standard spectral sequences converging to
$H^*(\Cc^\bullet_{\Tot},D)$ \cite[Ch.2.4]{McCl}. The first one,
horizontal:
\[
\Eh_r^{*,*},\qquad \dho_r\colon\Eh_r^{k,l}\to\Eh_r^{k+1-r,l+r}
\]
computes horizontal cohomology first, then vertical cohomology.
The second, vertical,
\[
\Ev_r^{*,*},\qquad \dv_r\colon\Ev_r^{k,l}\to\Ev_r^{k+r,l+1-r}
\]
computes vertical cohomology first, then horizontal.

\begin{lemma}\label{lemmaTotIsSheaf}
$H^k(\Cc^\bullet_{\Tot},D)\cong H^k(S;\hh_0\otimes\I)$.
\end{lemma}

\begin{proof}
Consider the horizontal spectral sequence:
\begin{gather*}\label{eqEhorizon}
\Eh_1^{k,l}=H^k(\Cc^l(S,\R_k),\dho)=
\begin{cases}
\Cc^l(S;\hh_0\otimes\I), k = 0;\\
0, \mbox{ o.w.}
\end{cases}\\
\Eh_2^{k,l}=\begin{cases}H^l(S;\hh_0\otimes\I),\mbox{ if } k=0;
\\0,\mbox{ o.w.}
\end{cases}
\end{gather*}
Spectral sequence $\Eh_*^{*,*}$ collapses at the second term and
the statement follows.
\end{proof}

\begin{lemma}\label{lemmaTotIsCosheaf}
$H^k(\Cc^\bullet_{\Tot},D)\cong H_{n-1-k}(S;\Pih)$.
\end{lemma}

\begin{proof}
Consider the vertical spectral sequence. It starts with
\[
\Ev_1^{k,l}\cong H^l(S;\R_k)=\bigoplus_{\dim I=-k}H^l(S;\R_I).
\]
Similar to example \ref{exSheafLocTens} we get
\[
H^l(S;\R_I)=H^l(S;\hh_0\otimes\ups{I}^{\Pi_I})=
H^{l-|I|}(\lk_SI;\hh_0|_{\lk I}\otimes\Pi_I)
\]
The restriction of $\hh_0$ to $\lk_SI\subset S$ coincides with the
structure sheaf of $F_I$ and by \eqref{eqCollapsingAtCMFace} we
have a collapsing
\[
H^{l-|I|}(\lk_SI;\hh_0|_{\lk I}\otimes\Pi_I)
\stackrel{\cong}{\Rightarrow} H_{n-1-l}(F_I)\otimes\Pi_I
\]
A proper face $F_I$ is acyclic, thus
\[
H^l(S;\R_I)\cong\begin{cases} \Pi_I,\mbox{ if }l=n-1,\\
0,\mbox{ if }l\neq n-1.
\end{cases}
\]
The maps $\theta_{I>I'}$ induce the isomorphisms $H_*(F_I)\to
H_*(F_{I'})$ which assemble in commutative squares
\[
\xymatrix{
H^{n-1}(S;\R_I)\ar@{=>}[r]^(0.7){\cong}\ar@{->}[d]^{\theta_{I>I'}^*}&
\Pi_I\ar@{^{(}->}[d]\\
H^{n-1}(S;\R_{I'})\ar@{=>}[r]^(0.7){\cong}&\Pi_{I'}}
\]
Thus the first term of vertical spectral sequence is identified
with the chain complex of cosheaf $\Pih$:
\[
\Ev_1^{k,l}=\begin{cases} \Cc_{-k}(S;\Pih),\mbox{ if }l=n-1;\\
0,\mbox{ if }l\neq n-1.
\end{cases}\qquad
\Ev_2^{k,l}=\begin{cases} H_{-k}(S;\Pih),\mbox{ if }l=n-1;\\
0,\mbox{ if }l\neq n-1.
\end{cases}
\]
The spectral sequence $\Ev_*^{k,l}\Rightarrow
H^{k+l}(\Ccq^\bullet_{\Tot},D)$ collapses at the second page.
Lemma proved.
\end{proof}

Theorem \ref{thmDuality} follows from lemmas \ref{lemmaTotIsSheaf}
and \ref{lemmaTotIsCosheaf}.
\end{proof}

\subsection{Extending duality to exact sequences}

Theorem \ref{thmDuality} can be refined:

\begin{stm}\label{stmSeqOfSheavesDual}
The short exact sequence of sheaves
\begin{equation}\label{eqShortExIdeal}
0\rightarrow\hh_0\otimes\I\rightarrow\hh_0\otimes\La \rightarrow
\hh_0\otimes (\La/\I)\rightarrow 0
\end{equation}
and the short exact sequence of cosheaves
\[
0\rightarrow\Pih\rightarrow\Lah \rightarrow \Lah/\Pih\rightarrow 0
\]
induce isomorphic long exact sequences in (co)homology:
\begin{equation}\label{eqTwoSeqIsom}
\xymatrix{H^i(S;\hh_0\otimes\I)\ar@{->}[r]&
H^i(S;\hh_0\otimes\La)\ar@{->}[r]&H^i(S;\hh_0\otimes(\La/\I))\ar@{->}[r]&
H^{i+1}(S;\hh_0\otimes\I)\\
H_{n-1-i}(S;\Pih)\ar@{->}[r]\ar@{<->}[u]_{\cong}&
H_{n-1-i}(S;\Lah)\ar@{->}[r]\ar@{<->}[u]_{\cong}&H_{n-1-i}(S;(\Lah/\Pih))
\ar@{->}[r]\ar@{<->}[u]_{\cong}&H^{n-2-i}(S;\Pih)\ar@{<->}[u]_{\cong}}
\end{equation}
\end{stm}

\begin{proof}
The proof goes essentially the same as in Theorem
\ref{thmDuality}. Denote sequence \eqref{eqShortExIdeal} by
$\Iseq$. For each $\varnothing\neq I\in S$ consider the short
exact sequence of sheaves:
\[
\Rseq_I\quad\colon\quad
0\rightarrow\R_I\rightarrow\hh_0\otimes\La\rightarrow
\hh_0\otimes\left(\La/\ups{I}^{\Pi_I}\right)\rightarrow 0
\]
and define
\[
\Rseq_{-k}=\bigoplus_{\dim I=k}\Rseq_I
\]
One can view $\Iseq$, $\Rseq_I$ and $\Rseq_{-k}$ as the objects in
a category of complexes. As before, we can form the sequence
\begin{equation}\label{eqLongSeqOfSeqs}
\Rseq_{\bullet}\colon\qquad\ldots\longrightarrow\Rseq_{-2}\stackrel{d_H}{\longrightarrow}
\Rseq_{-1}\stackrel{d_H}{\longrightarrow}\Rseq_{0}\stackrel{\eta}{\longrightarrow}
\Iseq\longrightarrow 0,
\end{equation}
which happens to be exact in all positions. This long exact
sequence (after forgetting the last term) generates the bicomplex
of short exact sequences (or the short exact sequence of
bicomplexes) $\Cseq_{\bullet}^{\bullet}$. By taking totalization
and considering standard spectral sequences we check that both
rows in \eqref{eqTwoSeqIsom} are isomorphic to the long exact
sequence of cohomology associated to $(\Cseq^{\bullet}_{\Tot},D)$.
\end{proof}

\subsection{Remark on duality}

In the manifold case (i.e. sheaf $\hh_0$ is isomorphic to $\ko$),
the proof of Theorem \ref{thmDuality} can be restated in more
conceptual terms. In this case the cellular version of Verdier
duality for manifolds \cite[Th.12.3]{Curry} asserts:
\[
H^i(S;\I) = H_{n-1-i}(S;\I),
\]
where the homology groups of a cellular sheaf are defined as
homology of global sections of projective sheaf resolution
\cite[Def.11.29]{Curry}. The sheaf
$\R_I=\hh_0\otimes\ups{I}^{\Pi_I}\cong \ups{I}^{\Pi_I}$ is
projective (\cite[Sec.11.1.1]{Curry}), thus
\eqref{eqLongSeqOfSheaves} is actually a projective resolution.
Due to the specific structure of this resolution, we have
$H_*(S;\I)\cong H_*(S;\Pih)$.

\section{Face vectors and ranks of border
components}\label{secHvectors}

In this section we prove Theorems \ref{thmBorderStruct},
\ref{thmBorderManif} and \ref{thmHtwoPrimes}.

\subsection{Preliminaries on face vectors}
First recall several standard definitions from combinatorial
theory of simplicial complexes and posets.

\begin{con}
Let $S$ be a pure simplicial poset, $\dim S=n-1$. Let
$f_i(S)=\{I\in S\mid \dim I=i\}$, $f_{-1}(S)=1$. The array
$(f_{-1},f_0,\ldots,f_{n-1})$ is called the $f$-vector of $S$. We
write $f_i$ instead of $f_i(S)$ since $S$ is clear from the
context. Let $f_S(t)$ be the generating polynomial: $f_S(t) =
\sum_{i\geqslant 0}f_{i-1}t^i$.

Define the $h$-numbers by the relation:
\begin{equation}\label{eqHvecDefin}
\sum_{i=0}^nh_i(S)t^i=\sum_{i=0}^nf_{i-1}t^i(1-t)^{n-i}=
(1-t)^nf_S\left(\dfrac{t}{1-t}\right).
\end{equation}

Let $b_i(S) = \dim H_i(S)$, $\br_i(S)=\dim \Hr_i(S)$,
$\chi(S)=\sum_{i=0}^{n-1}(-1)^ib_i(S)=
\sum_{i=0}^{n-1}(-1)^if_i(S)$ and
$\chir(S)=\sum_{i=0}^{n-1}\br_i(S)=\chi(S)-1$. Thus
$f_S(-1)=1-\chi(S)$. Also note that $h_{n}(S)=(-1)^{n-1}\chir(S)$.

Define $h'$- and $h''$-vectors by
\begin{equation}\label{eqDefHprim}
h_i'=h_i+{n\choose
i}\left(\sum_{j=1}^{i-1}(-1)^{i-j-1}\br_{j-1}(S)\right)\mbox{ for
} 0\leqslant i\leqslant n;
\end{equation}
\begin{equation}\label{eqdefHtwoprimes}
h_i'' = h_i'-{n\choose i}\br_{i-1}(S) = h_i+{n\choose
i}\left(\sum_{j=1}^{i}(-1)^{i-j-1}\br_{j-1}(S)\right)\mbox{ for }
0\leqslant i\leqslant n-1,
\end{equation}
and $h''_n=h'_n$. Note that
\begin{equation}\label{eqHnprime}
h'_n=h_n+\sum_{j=0}^{n-1}(-1)^{n-j-1}\br_{j-1}(S)= \br_{n-1}(S).
\end{equation}
\end{con}

\begin{stm}[Dehn--Sommerville relations]
If $S$ is Buchsbaum and $\dim\hh_0(I)~=~1$ for each $I\neq
\varnothing$, then
\begin{equation}\label{eqDehnSomQuasiManif}
h_i=h_{n-i}+(-1)^i{n\choose i}(1-(-1)^n-\chi(S)),
\end{equation}
or, equivalently:
\begin{equation}\label{eqDehnSomQuasiManifChir}
h_i=h_{n-i}+(-1)^i{n\choose i}(1+(-1)^n\chir(S)),
\end{equation}
If, moreover, $S$ is a homology manifold, then $h''_i=h''_{n-i}$.
\end{stm}

\begin{proof}
The first statement can be found e.g. in \cite{St} or
\cite[Thm.3.8.2]{BPnew}. Also see remark \ref{remDehnSomManif}
below. The second then follows from the definition of $h''$-vector
and Poincare duality \eqref{eqBuchManPoincS}
$b_i(S)=b_{n-1-i}(S)$.
\end{proof}


\begin{defin}
Let $S$ be Buchsbaum. For $i\geqslant 0$ consider
\[
\ft_i(S)=\sum_{I\in S,\dim I=i}\dim \Hr_{n-1-|I|}(\lk_SI) =
\sum_{I\in S,\dim I=i}\dim \hh_0(I).
\]
\end{defin}

If $S$ is a homology manifold, then $\ft_i(S)=f_i(S)$. For general
Buchsbaum complexes there is another formula connecting these
quantities.

\begin{prop}\label{propFtildeToF}
$f_S(t)=(1-\chi(S))+(-1)^n\sum_{k\geqslant
0}\ft_{k}(S)\cdot(-t-1)^{k+1}$.
\end{prop}

\begin{proof}
This follows from the general statement
\cite[Th.9.1]{MMP},\cite[Th.3.8.1]{BPnew}, but we provide an
independent proof for completeness. As stated in
\cite[Lm.3.7,3.8]{AyBuch} for simplicial complexes (and also not
difficult to prove for posets)
$\frac{d}{dt}f_S(t)=\sum_{v\in\ver(S)}f_{\lk v}(t)$, and, more
generally,
\[
\left(\frac{d}{dt}\right)^kf_S(t)=k!\sum_{I\in S, |I|=k}f_{\lk
I}(t).
\]
Thus for $k\geqslant 1$:
\begin{multline*}
f_S^{(k)}(-1)=k!\sum_{I\in S, |I|=k}(1-\chi(\lk_SI)) =\\=
k!\sum_{I\in S, |I|=k}(-1)^{n-|I|}\dim \Hr_{n-|I|-1}(\lk I) =
(-1)^{n-k}k!\ft_{k-1}(S).
\end{multline*}
Considering the Taylor expansion of $f_S(t)$ at $-1$:
\[
f_S(t)=f_S(-1)+\sum_{k\geqslant
1}\frac{1}{k!}f_S^{(k)}(-1)(t+1)^k= (1-\chi(S))+\sum_{k\geqslant
0}(-1)^{n-k-1}\ft_k(S)\cdot(t+1)^{k+1},
\]
finishes the proof.
\end{proof}

\begin{rem}\label{remDehnSomManif}
If $S$ is a manifold, then proposition \ref{propFtildeToF} implies
\[
f_S(t) = (1-(-1)^n-\chi(S))+(-1)^nf_S(-t-1),
\]
which is an equivalent form of Dehn--Sommerville relations
\eqref{eqDehnSomQuasiManif}.
\end{rem}

\begin{lemma}\label{lemmaHfromFt}
For Buchsbaum poset $S$ there holds
\[
\sum_{i=0}^nh_it^i = (1-t)^n(1-\chi(S)) + \sum_{k\geqslant
0}\ft_k(S)\cdot(t-1)^{n-k-1}.
\]
\end{lemma}

\begin{proof}
Substitute $t/(1-t)$ in proposition \ref{propFtildeToF} and use
\eqref{eqHvecDefin}.
\end{proof}

The coefficients of $t^i$ in lemma \ref{lemmaHfromFt} give the
relations
\begin{equation}\label{eqFtildeToH}
h_i(S)=(1-\chi(S))(-1)^i{n\choose i}+\sum_{k\geqslant
0}(-1)^{n-k-i-1}{n-k-1\choose i} \ft_k(S).
\end{equation}

\subsection{Ranks of $\Ex^1_{*,*}$}

Our goal is to compute the ranks of border groups $\dim
\Ex^{1+}_{q,q}$. The idea is very straightforward: statement
\ref{stmThmSpecConvint} describes the ranks of all groups
$\Ex^{1+}_{p,q}$ except for $p=q$; and the terms $\Ex^1_{p,q}$ are
known as well; thus $\dim \Ex^{1+}_{q,q}$ can be found by
comparing Euler characteristics. Note that the terms with $p=n$ do
not change when passing from $\Ex^1$ to $\Ex^{1+}$. Thus it is
sufficient to perform calculations with the truncated sequence
$\Ext$. By construction, $\Ex^1_{p,q}\cong \Ext^1_{p,q}\cong
C^{n-p-1}(S;\hh^X_q)$ for $p<n$. Thus lemma
\ref{lemmaTruncExtDescr} implies for $p<n$:
\[
\dim \Ex^1_{p,q} = \dim \Ext^1_{p,q}
=\sum_{|I|=n-p}\dim\hh_0(I)\cdot\dim (\Lambda/\I(I))^{(q)} =
{p\choose q}\cdot \ft_{n-p-1}(S).
\]
Let $\chi^1_q$ be the Euler characteristic of $q$-th row of
$\Ext^1_{*,*}$:
\begin{equation}\label{eqChi1def}
\chi_q^1=\sum_{p\leqslant
n-1}(-1)^p\dim\Ext^1_{p,q}=\sum_{p\leqslant n-1}(-1)^p{p\choose
q}\ft_{n-p-1}
\end{equation}

\begin{lemma}\label{lemmaChic1}
For $q\leqslant n-1$ there holds $\chi_q^1=(\chi(S)-1){n\choose
q}+(-1)^qh_q(S)$.
\end{lemma}

\begin{proof}
Substitute $i=q$ and $k=n-p-1$ in \eqref{eqFtildeToH} and combine
with \eqref{eqChi1def}.
\end{proof}

\subsection{Ranks of $\Ex^{1+}_{*,*}$}

By construction of the extra page,
$\Ex^{1+}_{p,q}\cong\Ext^2_{p,q}$ for $p<n$. Let $\chi^2_q$ be the
Euler characteristic of $q$-th row of $\Ext^2_{*,*}$:
\begin{equation}\label{eqChiE2def}
\chi^2_q = \sum_p(-1)^p\dim\Ext^2_{p,q}.
\end{equation}
Euler characteristics of first and second terms coincide:
$\chi_q^2=\chi_q^1$. By statement \ref{stmThmSpecConvint},
$\dim\Ex^{1+}_{p,q}={n\choose q}b_p(S)$ for $q<p<n$. Lemma
\ref{lemmaChic1} yields
\[
(-1)^q\dim\Ex^{1+}_{q,q}+\sum_{p=q+1}^{n-1}(-1)^p{n\choose q}
b_p(S)=(\chi(S)-1){n\choose q}+(-1)^qh_q.
\]
By taking into account obvious relations between reduced and
non-reduced Betti numbers and equality
$\chi(S)=\sum_{p=0}^{n-1}b_p(S)$, this proves
Theorem~\ref{thmBorderStruct}.

\subsection{Manifold case}

If $X$ is a homology manifold, then Poincare duality
$b_i(S)=b_{n-i}(S)$ and Dehn--Sommerville relations
\eqref{eqDehnSomQuasiManifChir} imply
\begin{multline*}
\dim\Ex^{1+}_{q,q}=h_q+{n\choose q}\sum_{p=0}^q(-1)^{p+q}\br_p=\\
=h_q-(-1)^q{n\choose q}+{n\choose
q}\sum_{p=0}^q(-1)^{p+q}b_p=\\
=h_q-(-1)^q{n\choose q}+{n\choose
q}\sum_{p=n-1-q}^{n-1}(-1)^{n-1-p+q}b_p=\\
=h_{n-q}+(-1)^q{n\choose
q}\left[-(-1)^n+(-1)^n\chi+\sum_{p=n-1-q}^{n-1}(-1)^{n-1-p}b_p\right]=\\
=h_{n-q}+(-1)^q{n\choose q}
\left[-(-1)^n+\sum_{p=0}^{n-q-2}(-1)^{p+n}b_p \right].
\end{multline*}

The final expression in brackets coincides with
$\sum_{p=-1}^{n-q-2}(-1)^{p+n}\br_p(S)$ whenever the summation is
taken over nonempty set, i.e. for $q<n-1$. Thus
$\dim\Ex^{1+}_{q,q}=h'_{n-q}$ for $q<n-1$. In the case $q=n-1$ we
get $\dim\Ex^{1+}_{n-1,n-1}=h_1+{n\choose n-1}=h'_1+n$. This
proves part (1) of Theorem \ref{thmBorderManif}.

Part (2) follows easily. Indeed, for $q=n$:
\[
\dim\Ex^2_{n,n}=\dim\Ex^1_{n,n}={n\choose n}\dim H_n(Q,\dd
Q)=1=h'_0
\]
For $q=n-1$:
\[
\dim \Ex^2_{n-1,n-1}=\dim\Ex^{1+}_{n-1,n-1}-{n\choose n-1}\dim\im
\delta_n=h'_1.
\]
If $q\leqslant n-2$, then $\Ex^2_{q,q}=\Ex^{1+}_{q,q}$, and the
statement follows from part (1).

\subsection{Cone case}

If $Q = P(S) \cong \cone|S|$, then the map $\delta_i\colon
H_i(Q,\dd Q)\to \Hr_{i-1}(\dd Q)$ is an isomorphism. Thus for
$q\leqslant n-1$ statement \ref{stmThmSpecConvint} implies
\[
\dim \Ex^{\infty}_{q,q}= \dim \Ex^{1+}_{q,q}-{n\choose q}\dim
H_{q+1}(Q,\dd Q) = \dim \Ex^{1+}_{q,q}-{n\choose q}\br_q(S).
\]
By Theorem \ref{thmBorderStruct} this expression is equal to
\[
h_q(S)+{n\choose q}
\left[\sum_{p=0}^q(-1)^{p+q}\br_p(S)\right]-{n\choose q}\br_q(S) =
h_q(S)+{n\choose q}\sum_{p=0}^{q-1}(-1)^{p+q}\br_p(S)=h''_{q}(S).
\]
The case $q=n$ follows from \eqref{eqHnprime}. Indeed, the term
$\Ex^{1+}_{n,n}$ survives, thus:
\[
\dim \Ex^{\infty}_{n,n} = {n\choose n}\dim H_n(Q,\dd Q)=b_{n-1}(S)
= h_{n}'(S)=h_n''(S).
\]
This proves Theorem \ref{thmHtwoPrimes}.

\section{Geometry of equivariant cycles}\label{secEquivarGeometry}

\subsection{Orientations}
In this section we restrict to the case when $Q$ is a nice
manifold with corners, $X=(Q\times T^n)/\simc$ is a manifold with
locally standard torus action, $\lambda$ --- a characteristic map
over $\Zo$ defined on the poset $S=S_Q$. As before, suppose that
all proper faces of $Q$ are acyclic and orientable and $Q$ itself
is orientable. The subset $X_I$, $I\neq \varnothing$ is a
submanifold of $X$, preserved by the torus action; $X_I$ is called
a face manifold, $\codim X_I=2|I|$. Submanifolds $X_{\{i\}}$,
corresponding to vertices $i\in\ver(S)$ are called characteristic
submanifolds, $\codim X_{\{i\}}=2$.

Fix arbitrary orientations of the orbit space $Q$ and the torus
$T^n$. This defines the orientation of $Y=Q\times T^n$ and
$X=Y/\simc$. Also choose an omniorientation, i.e. orientations of
all characteristic submanifolds $X_{\{i\}}$. The choice of
omniorientation defines characteristic values $\omega_i\in
H_1(T^n;\Zo)$ without ambiguity of sign (recall that previously
they were defined only up to units of $\ko$). To perform
calculations with the spectral sequences $\Ex$ and $\Ey$ we also
need to orient faces of $Q$.

\begin{lemma}[Convention]
The orientation of each simplex of $S$ (i.e. the sign convention
on $S$) defines the orientation of each face $F_I\subset Q$.
\end{lemma}

\begin{proof}
Suppose that $I\in S$ is oriented. Let $i_1,\ldots,i_{n-q}$ be the
vertices of $I$, listed in a positive order (this is where the
orientation of $I$ comes in play). The corresponding face $F_I$
lies in the intersecion of facets $F_{i_1},\ldots,F_{i_{n-q}}$.
The normal bundles $\nu_{i}$ to facets $F_{i}$ have natural
orientations, in which inward normal vectors are positive. Orient
$F_I$ in such way that
$T_xF_I\oplus\nu_{i_i}\oplus\ldots\oplus\nu_{i_{n-q}}\cong T_xQ$
is positive.
\end{proof}

Thus there are distinguished elements $[F_I]\in H_{\dim
F_I}(F_I,\dd F_I)$. One checks that for $I<_1J$ the maps
\[
m^0_{I,J}\colon H_{\dim F_I}(F_I,\dd F_I)\to H_{\dim F_I-1}(\dd
F_I)\to H_{\dim F_J}(F_J,\dd F_J)
\]
(see \eqref{eqMattachingMap}) send $[F_I]$ to
$\inc{J}{I}\cdot[F_J]$. Thus the restriction maps $\hh_0(I<J)$
send $[F_I]$ to $[F_J]$ by the definition of $\hh_0$.

The choice of omniorientation and orientations of $I\in S$
determines the orientation of each orbit $T^n/T^{\lambda(I)}$ by
the following convention.

\begin{con}
Let $i_1,\ldots,i_{n-q}$ be the vertices of $I$, listed in a
positive order. Recall that $H_1(T^n/T^{\lambda(I)})$ is naturally
identified with $H_1(T^n)/\I(I)^{(1)}$. The basis
$[\gamma_1],\ldots,[\gamma_q]\in H_1(T^n/T^{\lambda(I)})$,
$[\gamma_l]=\gamma_l+\I(I)^{(1)}$ is defined to be positive if the
basis
$(\omega_{i_1},\ldots,\omega_{i_{n-q}},\gamma_1,\ldots,\gamma_q)$
is positive in $H_1(T^n)$. The orientation of $T^n/T^{\lambda(I)}$
determines a distinguished ``volume form''
$\Omega_I=\bigwedge_l[\gamma_l]\in H_q(T^n/T^{\lambda(I)};\Zo)$.
\end{con}

The omniorientation and the orientation of $S$ also determine the
orientation of each manifold $X_I$ in a similar way. All
orientations are compatible: $[X_I]=[F_I]\otimes\Omega_I$.

\subsection{Arithmetics of torus quotients}
Let us fix a positive basis $e_1,\ldots,e_n$ of the lattice
$H_1(T^n;\Zo)$. Let $(\lambda_{i,1},\ldots,\lambda_{i,n})$ be the
coordinates of $\omega_i$ in this basis for each $i\in \ver(S)$.
The following technical lemma will be used in subsequent
computations.

\begin{lemma}\label{lemmaToricCoefficient}
Let $I\in S$, $I\neq\varnothing$ be a simplex with vertices
$\{i_1,\ldots,i_{n-q}\}$ listed in a positive order. Let
$A=\{j_1<\ldots<j_q\}\subset [n]$ be a subset of indices and
$e_A=e_{j_1}\wedge\ldots\wedge e_{j_q}$ the corresponding element
of $H_q(T^n)$. Consider the map $\rho\colon T^n\to
T^n/T^{\lambda(I)}$. Then $\rho_*(e_A) = C_{A,I}\Omega_I\in
H_q(T^n/T^{\lambda(I)})$ with the constant:
\[
C_{A,I}=\sgn_A\det\left(\lambda_{i,j}\right)_{\begin{subarray}{l}
i\in \{i_1,\ldots,i_{n-q}\}\\j\in [n]\setminus A
\end{subarray}}
\]
where $\sgn_A=\pm1$ depends only on $A\subset[n]$.
\end{lemma}

\begin{proof}
Let $(b_l) =
(\omega_{i_1},\ldots,\omega_{i_{n-q}},\gamma_1,\ldots,\gamma_q)$
be a positive basis of lattice $H_1(T^n,\Zo)$. Thus $b_l=Ue_l$,
where the matrix $U$ has the form
\[
U=\begin{pmatrix}
\lambda_{i_1,1}&\ldots& \lambda_{i_{n-q},1} & * & *\\
\lambda_{i_1,2}&\ldots& \lambda_{i_{n-q},2} & * & *\\
\vdots &\ddots & \vdots & \vdots & \vdots\\
\lambda_{i_1,n}&\ldots& \lambda_{i_{n-q},n} & * & *
\end{pmatrix}
\]
We have $\det U=1$ since both bases are positive. Consider the
inverse matrix $V=U^{-1}$. Thus
\[
e_A=e_{j_1}\wedge\ldots\wedge
e_{j_q}=\sum_{M=\{\alpha_1<\ldots<\alpha_q\}\subset[n]}
\det\left(V_{j,\alpha}\right)_{\begin{subarray}{l} j\in
A\\\alpha\in M \end{subarray}}b_{\alpha_1}\wedge\ldots\wedge
b_{\alpha_q}.
\]
After passing to quotient $\Lambda\to \Lambda/\I(I)$ all summands
with $M\neq\{n-q+1,\ldots,n\}$ vanish. When
$M=\{n-q+1,\ldots,n\}$, the element $b_{n-q-1}\wedge\ldots\wedge
b_{n} = \gamma_1\wedge\ldots \wedge\gamma_q$ goes to $\Omega_I$.
Thus
\[
C_{A,I} = \det\left(V_{j,\alpha}\right)_{\begin{subarray}{l}j\in
A\\\alpha\in\{n-q+1,\ldots,n\}\end{subarray}}.
\]
Now apply Jacobi's identity which states the following (see e.g.
\cite[Sect.4]{BruShn}). Let $U$ be an invertible $n\times
n$-matrix, $V=U^{-1}$, $M,N\subset[n]$ subsets of indices,
$|M|=|N|=q$. Then
\[
\det\left(V_{r,s}\right)_{\begin{subarray}{l} r\in M\\s\in
N\end{subarray}}= \frac{\sgn_{M,N}}{\det U}
\det\left(U_{r,s}\right)_{\begin{subarray}{l}r\in [n]\setminus
N\\s\in [n]\setminus M\end{subarray}},
\]
where $\sgn_{M,N}=(-1)^{\sum_{r\in [n]\setminus N}r +\sum_{s\in
[n]\setminus M}s}$. In our case $N=\{n-q+1,\ldots,n\}$; thus the
sign depends only on $A\subset [n]$.
\end{proof}

\subsection{Face ring and linear system of parameters}

Recall the definition of a face ring of a simplicial poset $S$.
For $I_1,I_2\in S$ let $I_1\vee I_2\subset S$ denote the set of
least upper bounds, and $I_1\cap I_2\in S$ --- the intersection of
simplices (it is well-defined and unique if $I_1\vee
I_2\neq\varnothing$).

\begin{defin}
The face ring $\ko[S]$ is the quotient ring of $\ko[v_I\mid I\in
S]$, $\deg v_I = 2|I|$ by the relations
\[
v_{I_1}\cdot v_{I_2}=v_{I_1\cap I_2}\cdot\sum_{J\in I_1\vee
I_2}v_J,\qquad\quad v_{\varnothing}=1,
\]
where the sum over an empty set is assumed to be $0$.
\end{defin}

Characteristic map $\lambda$ determines the set of linear forms
$\{\theta_1,\ldots,\theta_n\}\subset \ko[S]$,
$\theta_j=\sum_{i\in\ver(S)} \lambda_{i,j}v_i$. If $J\in S$ is a
maximal simplex, $|J|=n$, then
\begin{equation}\label{eqMinorCharMap}
\mbox{the matrix }(\lambda_{i,j})_{\begin{subarray}{l} i\leqslant
J\\j\in [n]\end{subarray}}\mbox{ is invertible over }\ko
\end{equation}
by the $\sta{\ko}$-condition. Thus the sequence
$\{\theta_1,\ldots,\theta_n\}\subset \ko[S]$ is a linear system of
parameters in $\ko[S]$ (see, e.g.,\cite[lemma 3.5.8]{BPnew}). It
generates an ideal $(\theta_1,\ldots,\theta_n)\subset \ko[S]$
which we denote by $\Theta$.

The face ring $\ko[S]$ is an algebra with straightening law (see,
e.g. \cite[\S.3.5]{BPnew}). Additively it is freely generated by
the elements
\[
P_\sigma=v_{I_1}\cdot v_{I_2}\cdot\ldots\cdot v_{I_t},\qquad
\sigma=(I_1\leqslant I_2\leqslant\ldots\leqslant I_t).
\]

\begin{lemma}\label{lemmaFaceRingBasis}
The elements $[v_I]=v_I+\Theta$ additively generate
$\ko[S]/\Theta$.
\end{lemma}

\begin{proof}
Consider an element $P_\sigma$ with $|\sigma|\geqslant 2$. Using
relations in the face ring, we express $P_{\sigma}=v_{I_1}\cdot
\ldots\cdot v_{I_t}$ as $v_{i}\cdot v_{I_1\setminus i}\cdot
\ldots\cdot v_{I_t}$, for some vertex $i\leqslant I_1$. The
element $v_i$ can be expressed as $\sum_{i'\nleq I_t}a_{i'}v_{i'}$
modulo $\Theta$ according to \eqref{eqMinorCharMap} (we can
exclude all $v_i$ corresponding to the vertices of some maximal
simplex $J\supseteq I_t$). Thus $v_iv_{I_t}$ is expressed as a
combination of $v_{I_t'}$ for $I_t'>_1I_t$. Therefore, up to ideal
$\Theta$, the element $P_{\sigma}$ is expressed as a linear
combination of elements $P_{\sigma'}$ which have either smaller
length $t$ (in case $|I_1|=1$) or smaller $I_1$ (in case
$|I_1|>1$). By iterating this descending process, the element
$P_{\sigma}+\Theta\in \ko[S]/\Theta$ is expressed as a linear
combination of $[v_I]$.
\end{proof}

Note that the proof works for $\ko=\Zo$ as well.

\subsection{Linear relations on equivariant (co)cycles}

Let $H^*_T(X)$ be a $T^n$-equi\-va\-ri\-ant cohomology ring of
$X$. Any proper face of $Q$ is acyclic, thus has a vertex.
Therefore, there is the injective homomorphism
\[
\ko[S]\hookrightarrow H_T^*(X),
\]
which sends $v_I$ to the cohomology class, equivariant Poincare
dual to $[X_I]$ (see \cite[Lemma 6.4]{MasPan}). The inclusion of a
fiber in the Borel construction, $X\to X\times_TET^n$, induces the
map $H_T^*(X)\to H^*(X)$. The subspace $V$ of $H_*(X)$, Poincare
dual to the image of
\begin{equation}\label{eqFaceToCohomology}
g\colon\ko[S]\hookrightarrow H_T^*(X)\to H^*(X)
\end{equation}
is generated by the elements $[X_I]$, thus coincides with the
$\infty$-border: $V=\bigoplus_q\Ex^{\infty}_{q,q}\subset H_*(X)$.
Now let us describe explicitly the linear relations on $[X_I]$ in
$H_*(X)$. Note that the elements $[X_I]=[F_I]\otimes\Omega_I$ can
also be considered as the free generators of the $\ko$-module
\[
\bigoplus_q\Ex^1_{q,q}=\bigoplus_q\bigoplus_{|I|=n-q} H_q(F_I,\dd
F_I)\otimes H_q(T^n/T^{\lambda(I)}).
\]
The free $\ko$-module on generators $[X_I]$ is denoted by
$\langle[X_I]\rangle$.

\begin{prop}\label{propTwoTypesOfRels}
Let $C_{A,J}$ be the constants defined in lemma
\ref{lemmaToricCoefficient}. There are only two types of linear
relations on classes $[X_I]$ in $H_*(X)$:
\begin{enumerate}
\item For each $J\in S$, $|J|=n-q-1$, and $A\subset [n]$, $|A|=q$
there is a relation
\[
R_{J,A}=\sum_{I>_1J}\inc{I}{J}C_{A,I}[X_I]=0;
\]
\item Let $\beta$ be a homology class from $\im(\delta_{q+1}\colon
H_{q+1}(Q,\dd Q)\to H_q(\dd Q))\subseteq\Eqt^{\infty}_{q,0}$ for
$q\leqslant n-2$, and let $\sum_{|I|=n-q}B_I[F_I]\in \Eqt^1_{q,0}$
be a chain representing $\beta$. Then
\[
R'_{\beta,A}=\sum_{|I|=n-q}B_IC_{A,I}[X_I]=0.
\]
\end{enumerate}
\end{prop}

\begin{proof}
This follows from the structure of the map
$f_*\colon\Eq^*_{*,*}\times H_*(T^n)\to\Ex^*_{*,*}$, lemma
\ref{lemmaToricCoefficient} and Theorem
\ref{thmTwoSpecSeqGeneral}. Relations on $[X_I]$ appear as the
images of the differentials hitting $\Ex^r_{q,q}$, $r\geqslant 1$.
Relations of the first type, $R_{J,A}$, are the images of
$\dx^1\colon \Ex^1_{q+1,q}\to \Ex^1_{q,q}$. In particular,
$\bigoplus_q\Ex^2_{q,q}$ is identified with $\langle
[X_I]\rangle/\langle R_{J,A}\rangle$. Relations of the second type
are the images of higher differentials $\dx^r$, $r\geqslant 2$.
\end{proof}

Now we check that relations of the first type are exactly the
relations in the quotient ring $\ko[S]/\Theta$.

\begin{prop}
Let $\varphi\colon\langle[X_I]\rangle\to \ko[S]$ be the degree
reversing linear map, which sends $[X_I]$ to $v_I$. Then $\varphi$
descends to the isomorphism
\[
\tilde{\varphi}\colon \langle[X_I]\rangle/\langle
R_{J,A}\rangle\to \ko[S]/\Theta.
\]
\end{prop}

\begin{proof}
(1) First we prove that $\tilde{\varphi}$ is well defined. The
image of $R_{J,A}$ is the element
\[
\varphi(R_{J,A})=\sum_{I>_1J}\inc{I}{J}C_{A,I}v_I\in\ko[S].
\]
Let us show that $\varphi(R_{J,A})\in\Theta$. Let $s=|J|$, and
consequently, $|I|=s+1$, $|A|=n-s-1$. Let $[n]\setminus
A=\{\alpha_1<\ldots<\alpha_{s+1}\}$ and let $\{j_1,\ldots,j_s\}$
be the vertices of $J$ listed in a positive order. Consider
$s\times(s+1)$ matrix:
\[
D=\begin{pmatrix}
\lambda_{j_1,\alpha_1}&\ldots& \lambda_{j_1,\alpha_{s+1}}\\
\vdots &\ddots & \vdots\\
\lambda_{j_s,\alpha_1}&\ldots& \lambda_{j_s,\alpha_{s+1}}
\end{pmatrix}
\]
Denote by $D_l$ the square submatrix obtained from $D$ by deleting
$i$-th column and let $a_l=(-1)^{l+1}\det D_l$. We claim that
\[
\varphi(R_{J,A}) = \pm
v_J\cdot(a_1\theta_{\alpha_1}+\ldots+a_{s+1}\theta_{\alpha_{s+1}})
\]
Indeed, after expanding each $\theta_l$ as $\sum_{i\in
\ver(S)}\lambda_{i,l}v_i$, all elements of the form $v_Jv_{i}$
with $i<J$ cancel; others give $\inc{I}{J}C_{A,I}v_I$ for $I>_1J$
according to lemma \ref{lemmaToricCoefficient} and cofactor
expansions of determinants (the incidence sign arise from
shuffling columns). Thus $\tilde{\varphi}$ is well defined.

(2) $\tilde{\varphi}$ is surjective by lemma
\ref{lemmaFaceRingBasis}.

(3) The dimensions of both spaces are equal. Indeed, $\dim
\langle[X_I]\mid|I|=n-q\rangle/\langle R_{J,A}\rangle=\dim
\Ex^2_{q,q}=h'_{n-q}(S)$ by Theorem \ref{thmBorderManif}. But
$\dim(\ko[S]/\Theta)^{(n-q)}=h'_{n-q}(S)$ by Schenzel's theorem
\cite{Sch}, \cite[Ch.II,\S8.2]{St}, (or \cite[Prop.6.3]{NS} for
simplicial posets) since $S$ is Buchsbaum.

(4) If $\ko$ is a field, then we are done. This implies the case
$\ko=\Zo$ as well.
\end{proof}

In particular, this proposition describes the additive structure
of $\ko[S]/\Theta$ in terms of the natural additive generators
$v_I$. Poincare duality in $X$ yields

\begin{cor}\label{corKernel}
The map $g\colon \ko[S]\to H^*(X)$ factors through $\ko[S]/\Theta$
and the kernel of $\tilde{g}\colon \ko[S]/\Theta\to H^*(X)$ is
additively generated by the elements
\[
L'_{\beta,A}=\sum_{|I|=n-q}B_IC_{A,I}v_I
\]
where $q\leqslant n-2$, $\beta\in\im(\delta_{q+1}\colon
H_{q+1}(Q,\dd Q)\to H_q(\dd Q))$, $\sum_{|I|=n-q}B_I[F_I]$ is a
cellular chain representing $\beta$, and $A\subset [n]$, $|A|=q$.
\end{cor}

\begin{rem}
The ideal $\Theta\subset \ko[S]$ coincides with the image of the
natural map $H^{>0}(BT^n)\to H^*_T(X)$. So the fact that $\Theta$
vanishes in $H^*(X)$ is not surprising. The interesting thing is
that $\Theta$ vanishes by geometrical reasons already in the
second term of the spectral sequence, while other relations in
$H^*(X)$ are the consequences of higher differentials.
\end{rem}

\begin{rem}
From the spectral sequence follows that the element
$L'_{\beta,A}\in\ko[S]/\Theta$ does not depend on the cellular
chain, representing $\beta$. All such chains produce the same
element in $\bigoplus_q\Ex^2_{q,q}=\langle[X_I]\rangle/\langle
R_{J,A}\rangle\cong \ko[S]/\Theta$. Theorem \ref{thmEX2structure}
also implies that the relations $\{L'_{\beta,A}\}$ are linearly
independent in $\ko[S]/\Theta$ when $\beta$ runs over some basis
of $\im\delta_{q+1}$ and $A$ runs over all subsets of $[n]$ of
cardinality $q$.
\end{rem}

\section{Examples and calculations}\label{secExamples}

\subsection{Quasitoric manifolds}\label{subsecQtoric}
Let $Q$ be $n$-dimensional simple polytope. Then $S=S_Q=\dd Q^*$
is the boundary of the polar dual polytope. In this case
$Q\cong\cone |S|$. Given a characteristic map $\lambda\colon
\ver(K)\to\Tt_n$ we construct a space $X=(Q\times T^n)/\simc$
which is a model of quasitoric manifold~\cite{DJ}. Poset $S$ is a
sphere thus $h''_i(S)=h'_i(S)=h_i(S)$. Since $\delta_n\colon
H_n(Q,\dd Q)\to H_{n-1}(\dd Q)$ is an isomorphism, Theorem
\ref{thmEX2structure} implies $\Ex^2_{p,q}=0$ for $p\neq q$. By
Theorems \ref{thmBorderStruct} and \ref{thmHtwoPrimes},
$\dim\Ex^2_{q,q} = h_q(S) = h_{n-q}(S)$. Spectral sequence
$\Ex^*_{*,*}$ collapses at its second term, thus $\dim H_{2q}(X) =
h_{q}(S) = h_{n-q}(S)$ for $0\leqslant q\leqslant n$ and $\dim
H_{2q+1}(X)=0$ which is well known. For bigraded Betti numbers
proposition \ref{propHomologyX} implies $H_{i,j}(X)=0$ if $i\neq
j$, and $\dim H_{i,i}(X)=h_i(S)$.

\subsection{Homology polytopes}\label{subsecHomolPoly}
Let $Q$ be a manifold with corners such that all its proper faces
as well as $Q$ itself are acyclic. Such objects were called
homology polytopes in \cite{MasPan}. In this case everything
stated in the previous paragraph remains valid, thus $\dim
H_{2q}(X) = h_{q}(S) = h_{n-q}(S)$ for $0\leqslant q\leqslant n$,
and $\dim H_{2q+1}(X)=0$ (see \cite{MasPan}).

\subsection{Origami toric manifolds}\label{subsecOrigami}
Origami toric manifolds appeared in differential geometry as
generalizations of symplectic toric manifolds (see
\cite{SGP},\cite{MP}). The original definition contains a lot of
subtle geometrical details and in most part is irrelevant to this
paper. Here we prefer to work with the ad hoc model, which
captures most essential topological properties of origami
manifolds.

\begin{defin}
Topological toric origami manifold $X^{2n}$ is a manifold with
locally standard action $T^n\curvearrowright X$ such that all
faces of the orbit space including $X/T$ itself are either
contractible or homotopy equivalent to wedges of circles.
\end{defin}

As before consider the canonical model. Let $Q^n$ be a nice
manifold with corners in which every face is contractible or
homotopy equivalent to a wedge of $\ddb$ circles. Every principal
$T^n$-bundle $Y$ over $Q$ is trivial (because $H^2(Q)=0$), thus
$Y=Q\times T^n$. Consider the manifold $X=Y/\simc$ associated to
some characteristic map over $\Zo$. Then $X$ is a topological
origami toric manifold.

To apply the theory developed in this paper we also assume that
all proper faces of $Q$ are acyclic (in origami case this implies
contractible) and $Q$ itself is orientable. Thus, in particular,
$Q$ is a Buchsbaum manifold. First, describe the exact sequence of
the pair $(Q,\dd Q)$. By Poincare--Lefchetz duality:
\[
H_q(Q,\dd Q)\cong H^{n-q}(Q)\cong\begin{cases} \ko,\mbox{ if } q=n;\\
H^1(\bigvee_{\ddb}S^1)\cong\ko^{\ddb},\mbox{ if }q=n-1;\\
0,\mbox{ otherwise.}
\end{cases}
\]
In the following let $m$ denote the number of vertices of $S$ (the
number of facets of $Q$). Thus $h_1'(S)=h_1(S)=m-n$. Consider
separately three cases:

\textbf{(1) $n=2$.} In this case $Q$ is an orientable
$2$-dimensional surface of genus $0$ with $\ddb+1$ boundary
components. Thus $\dd Q$ is a disjoint union of $\ddb+1$ circles
and long exact sequence in homology has the form:
\begin{multline*}
\overset{\substack{0\\\parallel}}{H_2(Q)}\longrightarrow
\overset{\substack{\ko\\\parallel}}{H_2(Q,\dd Q)}
\stackrel{\delta_2}{\longrightarrow}
\overset{\substack{\ko^{\ddb+1}\\\parallel}}{H_1(\dd
Q)}\longrightarrow
\overset{\substack{\ko^{\ddb}\\\parallel}}{H_1(Q)}\stackrel{0}{\longrightarrow}\\
\stackrel{0}{\longrightarrow}
\underset{\substack{\parallel\\\ko^{\ddb}}}{H_1(Q,\dd
Q)}\stackrel{\delta_1}{\longrightarrow}
\underset{\substack{\parallel\\\ko^{\ddb+1}}}{H_0(\dd
Q)}\longrightarrow
\underset{\substack{\parallel\\\ko}}{H_0(Q)}\longrightarrow
\underset{\substack{\parallel\\0}}{H_0(Q,\dd Q)}
\end{multline*}

The second term $\Ex^2_{*,*}$ of spectral sequence for $X$ is
given by Theorem \ref{thmEX2structure}. It is shown on a figure
below (only ranks are written to save space).

\def\sseqgridstyle{\ssgridcrossword}
\sseqxstart=0 \sseqystart=-1 \sseqentrysize=1.2cm \sseqxstep=1
\sseqystep=1

\begin{center}
\begin{sseq}{3}{4}
\ssdropcircled{\ddb+1}\ssname{p00} \ssmoveto 1 1 \ssdrop{m-2}
\ssmoveto 1 0 \ssdrop{\ddb} \ssmoveto 2 2 \ssdrop{1} \ssmoveto 2 1
\ssdrop{\ddb} \ssmoveto 2 0 \ssdrop{2\ddb} \ssmoveto 2 {-1}
\ssdrop{\ddb} \ssname{p2m1}

\ssgoto{p2m1} \ssgoto{p00} \ssstroke \ssarrowhead
\end{sseq}
\end{center}

The only nontrivial higher differential is $d^2\colon
\Ex^2_{2,-1}\to\Ex^2_{0,0}$; it coincides with the composition of
$\delta_1\otimes \id_{H_0(T^2)}$ and injective map $f_*^2\colon
H_0(P)\otimes H_0(T^2)\to \Ex^2_{0,0}$. Thus $d^2$ is injective,
and $\dim \Ex^{\infty}_{2,2}= \dim \Ex^{\infty}_{0,0} = 1$; $\dim
\Ex^{\infty}_{2,1}= \dim \Ex^{\infty}_{1,0}=\ddb$; $\dim
\Ex^{\infty}_{1,1}= m-2$; $\dim \Ex^{\infty}_{2,0}= 2\ddb$.
Finally,
\[
\dim H_i(X)=\begin{cases} 1,\mbox{ if }i=0,4;\\
\ddb,\mbox{ if }i=1,3;\\
m-2+2\ddb,\mbox{ if }i=2.
\end{cases}
\]
This coincides with the result of computations in \cite{PodSar},
concerning the same object. This result can be obtained simply by
proposition \ref{propHomologyX}: $\dim H_{0,0}(X)=\dim
H_{2,2}(X)=1$, $\dim H_{1,0}(X)=\dim H_{1,2}(X)=\ddb$, $\dim
H_{2,2}(X)=m-2+2\ddb$.

\textbf{(2) $n=3$.} In this case the exact sequence of $(Q,\dd Q)$
splits in three essential parts:
\begin{gather*}
\underset{\substack{\parallel\\0}}{H_3(Q)}\longrightarrow
\underset{\substack{\parallel\\\ko}}{H_3(Q,\dd Q)}
\stackrel{\delta_3}{\longrightarrow} H_2(\dd Q)\longrightarrow
\underset{\substack{\parallel\\0}}{H_2(Q)}\\
\underset{\substack{\parallel\\0}}{H_2(Q)}\longrightarrow
\underset{\substack{\parallel\\\ko^{\ddb}}}{H_2(Q,\dd Q)}
\stackrel{\delta_2}{\longrightarrow} H_1(\dd Q)\longrightarrow
\underset{\substack{\parallel\\\ko^{\ddb}}}{H_1(Q)}\longrightarrow
\underset{\substack{\parallel\\0}}{H_1(Q,\dd Q)}\\
\underset{\substack{\parallel\\0}}{H_1(Q,\dd
Q)}\stackrel{\delta_1}{\longrightarrow} H_0(\dd Q)\longrightarrow
\underset{\substack{\parallel\\\ko}}{H_0(Q)}\longrightarrow
\underset{\substack{\parallel\\0}}{H_0(Q,\dd Q)}
\end{gather*}

By Theorems \ref{thmEX2structure}, \ref{thmBorderManif},
$\Ex^2_{p,q}$ has the form

\begin{center}
\begin{sseq}{4}{5}
\ssdrop{h'_3} \ssmoveto 1 0 \ssdrop{2\ddb} \ssname{p10} \ssmoveto
1 1 \ssdropcircled{h_2'} \ssname{p11} \ssmoveto 2 2 \ssdrop{h_1'}
\ssmoveto 3 3 \ssdrop{1} \ssmoveto 3 2 \ssdrop{\ddb} \ssmoveto 3 1
\ssdrop{3\ddb} \ssmoveto 3 0 \ssdrop{3\ddb} \ssname{p30} \ssmoveto
3 {-1} \ssdrop{\ddb} \ssname{p3m1}

\ssmoveto 2 0 \ssdrop{0} \ssmoveto 2 1 \ssdrop{0}

\ssgoto{p30} \ssgoto{p11} \ssstroke \ssarrowhead

\ssgoto{p3m1} \ssgoto{p10} \ssstroke \ssarrowhead
\end{sseq}
\end{center}

There are two nontrivial higher differentials: $d^2\colon
\Ex^2_{3,0}\to \Ex^2_{1,1}$ and $d^2\colon
\Ex^2_{3,-1}\to\Ex^2_{1,0}$; both are injective. Thus $\dim
\Ex^{\infty}_{3,3}= \dim \Ex^{\infty}_{0,0} = 1$; $\dim
\Ex^{\infty}_{3,1}= \dim \Ex^{\infty}_{1,0}=\ddb$; $\dim
\Ex^{\infty}_{2,2}=h'_1$; $\dim \Ex^{\infty}_{3,1}=3\ddb$; $\dim
\Ex^{\infty}_{1,1}=h'_2-3\ddb$. Therefore,
\[
\dim H_i(X)=\begin{cases} 1,\mbox{ if }i=0,6;\\
\ddb,\mbox{ if }i=1,5;\\
h'_1+3\ddb,\mbox{ if }i=4;\\
h'_2-3\ddb,\mbox{ if }i=2;\\
0,\mbox{ if }i=3.
\end{cases}
\]

\textbf{(3) $n\geqslant 4$.} In this case lacunas in the exact
sequence for $(Q,\dd Q)$ imply that $\delta_i\colon H_i(Q,\dd
Q)\to H_{i-1}(\dd Q)$ is an isomorphism for $i=n-1,n$, and is
trivil otherwise. We have
\[
H_i(\dd Q)\cong\begin{cases} H_n(Q,\dd Q)\cong\ko,\mbox{ if } i=n-1;\\
H_{n-1}(Q,\dd Q)\cong\ko^{\ddb},\mbox{ if } i=n-2;\\ H_1(Q)\cong
\ko^{\ddb}
\mbox{ if } i=1;\\
H_0(Q)\cong\ko,\mbox{ if } i=0;\\
0,\mbox{ o.w.}
\end{cases}
\]

By Theorems \ref{thmEX2structure} and \ref{thmBorderManif},
$\Ex^2_{p,q}$ has the form

\sseqxstep=10 \sseqystep=10

\begin{center}
\begin{sseq}{7}{8}
\ssdrop{h'_n} \ssmoveto 1 0 \ssdrop{{n\choose 0}\ddb} \ssmoveto 1
1 \ssdrop{h_{n-1}'} \ssmoveto 2 2 \ssdrop{\iddots} \ssmoveto 3 3
\ssdrop{h_3'} \ssmoveto 4 4 \ssdropcircled{h_2'}\ssname{p44}
\ssmoveto 4 3 \ssdrop{{n\choose n-3}\ddb}\ssname{p43} \ssmoveto 4
2 \ssdrop{\vdots} \ssmoveto 4 1 \ssdrop{{n\choose
1}\ddb}\ssname{p41} \ssmoveto 4 0 \ssdrop{{n\choose 0}\ddb}
\ssname{p40} \ssmoveto 5 5 \ssdrop{h_1'} \ssmoveto 6 6 \ssdrop{1}
\ssmoveto 6 5 \ssdrop{{n\choose n} \ddb} \ssmoveto 6 4
\ssdrop{{n\choose n-1} \ddb} \ssmoveto 6 3 \ssdrop{{n\choose n-2}
\ddb}\ssname{p63} \ssmoveto 6 2 \ssdrop{{n\choose n-3}
\ddb}\ssname{p62} \ssmoveto 6 1 \ssdrop{\vdots} \ssmoveto 6 0
\ssdrop{{n\choose 1} \ddb}\ssname{p60} \ssmoveto 6 {-1}
\ssdrop{{n\choose 0} \ddb} \ssname{p6m1}

\ssmoveto 5 4 \ssdrop{0} \ssmoveto 5 3 \ssdrop{0} \ssmoveto 5 1
\ssdrop{0} \ssmoveto 5 0 \ssdrop{0} \ssmoveto 3 1 \ssdrop{0}
\ssmoveto 3 0 \ssdrop{0} \ssmoveto 2 0 \ssdrop{0} \ssmoveto 2 1
\ssdrop{0}

\ssgoto{p63} \ssgoto{p44} \ssstroke \ssarrowhead \ssgoto{p62}
\ssgoto{p43} \ssstroke \ssarrowhead \ssgoto{p60} \ssgoto{p41}
\ssstroke \ssarrowhead \ssgoto{p6m1} \ssgoto{p40} \ssstroke
\ssarrowhead
\end{sseq}
\end{center}

Thus we get: $\dim\Ex^{\infty}_{q,q}= h'_{n-q}$, if $q\neq n-2$;
$\dim\Ex^{\infty}_{n-2,n-2}= h'_2-{n\choose 2}\ddb$ if $q=n-2$;
$\dim\Ex^{\infty}_{n,n-1}=\dim\Ex^{\infty}_{1,0}=\ddb$;
$\dim\Ex^{\infty}_{n,n-2}=n\ddb$. Finally, by proposition
\ref{propHomologyX}, $\dim H_{1,0}(X)=\dim H_{n-1,n}(X)=\ddb$,
$\dim H_{n-1,n-1}(X)=h'_1+n\ddb$, $\dim
H_{n-2,n-2}(X)=h'_2-{n\choose 2}\ddb$, and $\dim
H_{i,i}(X)=h'_{n-i}$ for $i\neq n-1,n-2$.

The differential hitting the marked position produces additional
relations (of the second type) on the cycles $[X_I]\in
H^{2n-4}(X)$. These relations are described explicitly by
proposition \ref{propTwoTypesOfRels}. Dually, this consideration
shows that the map $\ko[S]/\Theta\to H^*(X)$ has a nontrivial
kernel only in degree $4$. The generators of this kernel are
described by corollary \ref{corKernel}.

\section{Concluding remarks}\label{secConclusion}

Several questions concerning the subject of this paper are yet to
be answered.

\begin{enumerate}

\item Of course, the main question which remains open is
the structure of multiplication in the cohomology ring $H^*(X)$.
The border module $\bigoplus_q\Ex^{\infty}_{q,q}\subset H_*(X)$
represents an essential part of homology; the structure of
multiplication on the corresponding subspace in cohomology can be
extracted from the ring homomorphism $\ko[S]/\Theta\to H^*(X)$.
Still there are cocycles which do not come from $\ko[S]$ and their
products should be described separately. Proposition
\ref{propHomologyX} suggests, that some products can be described
via the multiplication in $H^*(Q\times T^n)\cong H^*(Q)\otimes
H^*(T^n)$. This requires further investigation.

\item It is not clear yet, if there is a torsion in the border module
$\bigoplus_q\Ex^{\infty}_{q,q}$ in case $\ko=\Zo$. Theorems
\ref{thmBorderStruct},\ref{thmBorderManif},\ref{thmHtwoPrimes}
describe only the rank of the free part of this group, but the
structure (and existence) of torsion remains open. Note that the
homology of $X$ itself can have a torsion. Indeed, the groups
$H_*(Q)$, $H_*(Q,\dd Q)$ can contain arbitrary torsion, and these
groups appear in the description of $H_*(X)$ by proposition
\ref{propHomologyX}.

\item Corollary \ref{corKernel} describes the kernel of the map
$\ko[S]/\Theta\to H^*(X)$. It seems that the elements of this
kernel lie in a socle of $\ko[S]/\Theta$, i.e. in a submodule
$\{x\in \ko[S]/\Theta\mid (\ko[S]/\Theta)^{+}x=0\}$. The existence
of such elements is guaranteed in general by the Novik--Swartz
theorem \cite{NS}. If the relations $L'_{\beta,A}$ do not lie in a
socle, their existence would give refined inequalities on
$h$-numbers of Buchsbaum posets.

\item Theorem \ref{thmDuality} establish certain connection between the
sheaf of ideals generated by linear elements and the cosheaf of
ideals generated by exterior products. This connection should be
clarified and investigated further. In particular, statement
\ref{stmSeqOfSheavesDual} can probably lead to the description of
homology for the analogues of moment-angle complexes, i.e. the
spaces of the form $X=Y/\simc$, where $Y$ is an arbitrary
principal $T^\ld$-bundle over $Q$.

\item There is a hope, that the argument of section
\ref{SecKeyLemma} involving two spectral sequences for a sheaf
resolution can be generalized to non-Buchsbaum case.

\item The real case, when $T^\ld$ is replaced by $\Zt^\ld$, can,
probably, fit in the same framework.

\end{enumerate}

\section*{Acknowledgements}
I am grateful to prof. Mikiya Masuda for his hospitality and for
the wonderful environment with which he provided me in Osaka City
University. The problem of computing the cohomology ring of toric
origami manifolds, which he posed in 2013, was a great motivation
for this work (and served as a good setting to test working
hypotheses). Also I thank Shintaro Kuroki from whom I knew about
$h'$- and $h''$-vectors and their possible connection to torus
manifolds.

\end{document}